\documentclass[11pt]{article}
\usepackage{epsfig}
\usepackage{color}
\usepackage{amsmath}
\usepackage{amssymb}  %
\usepackage{comment}
\usepackage{pgfplots}
\usepackage{mathdots}
\usepackage{tikz}
\usepackage{url}
\usepackage{multirow}

\newfont{\Bb}{msbm10 scaled\magstep0}

\newtheorem{lemma}{Lemma}[section]

\newtheorem{remark}{Remark}[section]

\usepackage{algorithm}
\usepackage{algorithmic}

\newenvironment{proof}[1][\textit{Proof}]{\begin{trivlist}
\item[\hskip \labelsep {\bfseries #1}]}{\end{trivlist}}

\newcommand{\qed}{\nobreak \ifvmode \relax \else
      \ifdim\lastskip<1.5em \hskip-\lastskip
\vskip-1.2em
   \hskip10em plus0em minus0.5em \fi \nobreak
      \vrule height 0.4em width 0.3em depth0.25em}

\title{Multi-fidelity error estimation accelerates greedy model reduction of complex dynamical systems}\author{Lihong Feng \thanks{Max Planck Institute for Dynamics of Complex Technical Systems, Sandtorstrasse 1, 39106 Magdeburg, Germany~{\tt feng@mpi-magdeburg.mpg.de}}
, Luigi Lombardi\thanks{Luigi Lombardi is with Micron Semiconductor, 67051 Avezzano, Italy.~{\tt luigilombardi89@gmail.com}}
, Giulio Antonini\thanks{Giulio Antonini is with the UAq EMC Laboratory,
Department of Industrial and Information
Engineering and Economics, University of L'Aquila, I-67100 L'Aquila, Italy.~{\tt  giulio.antonini@univaq.it}}
, and Peter Benner\thanks{Max Planck Institute for Dynamics of Complex Technical Systems, Magdeburg, Germany and Fakultät für Mathematik, Otto-von-Guericke-Universität Magdeburg, Germany.~{\tt benner@mpi-magdeburg.mpg.de}}
}

\begin{document}



\maketitle

\begin{abstract}
Model order reduction usually consists of two stages: the offline stage and the online stage. The offline stage is the expensive part that sometimes takes hours till the final reduced-order model is derived, especially when the original model is very large or complex. Once the reduced-order model is obtained, the online stage of querying the reduced-order model for simulation is very fast and often real-time capable. This work concerns a strategy to significantly speed up the offline stage of model order reduction for large and complex systems. In particular, it is successful in accelerating the greedy algorithm that is often used in the offline stage for reduced-order model construction. We propose multi-fidelity error estimators and replace the high-fidelity error estimator in the greedy algorithm. Consequently, the computational complexity at each iteration of the greedy algorithm is reduced and the algorithm converges more than 3 times faster without incurring noticeable accuracy loss.  
\end{abstract}

\section{Introduction}
Model order reduction (MOR) has achieved much success in many areas of computational science with its capability of realizing real-time simulation and providing accurate results. Different MOR methods, their applications and the promising results they produce can be found in the survey papers~\cite{morAnt05, morBauBF14, morBenGW15} and books~\cite{morAntBG20, morBenCOetal17, morhandbookV1, morhandbookV2, morhandbookV3, morQuaMN16}. 

MOR needs an offline stage for constructing the ROM. For many intrusive MOR methods that are based on projection, the offline stage is usually realized via a greedy algorithm. The greedy algorithm is used to properly select {\it important} parameter samples that contribute most to the solution space. The offline computational time is basically the runtime of the greedy algorithm. For large-scale systems, the offline computation is expensive and the runtime is often longer than several hours even when run on a high-performance server. Sometimes, the system is not very large, for example, the number of degrees of freedom is only $O$($10^5$), but the system structure is complicated, so that the greedy algorithm still takes long time to converge. 

It is known that an efficient error estimator makes the greedy algorithm successful in producing an accurate ROM without running many iterations. Therefore, many efforts have been made in this direction to develop computable error estimators for different problems~\cite{morFenAB17, morFenB21, morGre05, morGre12, morGreMNetal07, morGreP05, morGruFH20, morHaaO08b, morHaaO11, morHaiOetal18, morRov03, morSchWH20, morSmeZP19, morVerPRetal03}. However, more attention has been paid to improve the effectivity or accuracy of the error estimator than to develop more efficient strategies to accelerate the greedy process~\cite{morCheFB22, morHaiOetal18, morSchWH20, morSmeZP19, morZhaFLetal15}. Recently, some techniques are proposed to improve the adaptivity of the greedy algorithm~\cite{morDavH15, morBenEEetal18b, morCheFB22, morTaiA15}. 

In~\cite{morCheFB22, morCheFdetal21a}, we proposed a surrogate model for error estimation, and proposed an adaptive greedy algorithm by alternatively using this surrogate error estimator and the original error estimator during the greedy algorithm. The focus in~\cite{morCheFB22, morCheFdetal21a} was to make the greedy process adaptive by starting from a coarse training set of a small number of parameter samples, and adaptively update the coarse training set with the aid of a surrogate error estimator. The original error estimator is computed only over the coarse training set, while the surrogate error estimator helps to pick out candidates of {\it important} parameter samples from a fine training set, which are then collected in the coarse training set. 

In this work, we emphasize the role of the surrogate error estimator and propose the concept of bi-fidelity error estimation and multi-fidelity error estimation. In fact, a bi-fidelity error estimation has been used in the adaptive greedy algorithm proposed in~\cite{morCheFB22, morCheFdetal21a} without being formally defined, i.e., the original (high-fidelity) error estimator, and the surrogate (low-fidelity) error estimator. To further improve the convergence speed of the greedy algorithm, we propose multi-fidelity error estimation built upon the bi-fidelity error estimation. Here, we use a more efficient high-fidelity error estimator than the two different high-fidelity error estimators used in~\cite{morCheFB22, morCheFdetal21a}. Although the proposed multi-fidelity error estimation is dependent on the original high-fidelity error estimator, the idea of using multi-fidelity error estimation is general and can be extended to develop multi-fidelity error estimation associated with other high-fidelity error estimators. 

Unlike the problems considered in~\cite{morCheFB22, morCheFdetal21a}, whose ROMs can be constructed by standard greedy algorithms within seconds to minutes, we consider in this work much more complicated problems. On the same computer, the standard greedy algorithm takes more than half a day to converge for such problems. By using the proposed multi-fidelity error estimator, the greedy algorithm achieves 4x speed-up and produces ROMs with little loss of accuracy. The speed-up is also higher than those reported in~\cite{morCheFB22, morCheFdetal21a} by using the bi-fidelity error estimation, which is usually 2x.

In the next section, we present the greedy algorithm in the standard form. Then we analyze some ingredients of the algorithm, which contribute most to the computational cost. Starting from those computationally expensive parts, we develop possible strategies to reduce the computational complexity in Section~\ref{sec:multifidelity}. As a consequence, it becomes clear that the resulting strategy develops a greedy algorithm with multi-fidelity error estimation. The proposed algorithm is then applied to large time-delay systems with many delays. The numerical tests on three large time-delay systems with more than 100 delays are demonstrated in Section~\ref{sec:results}. Conclusions are given in the end. 

\section{Standard greedy algorithm}
\label{sec:greedy_standard}

The standard greedy algorithm was first proposed for steady systems without time evolution. Then it was extended to POD-greedy for dynamical systems, which is used to construct the ROM using snapshots in the time domain. Later the greedy algorithm found its capability in adaptively choosing interpolation points for frequency-domain MOR methods~\cite{morFenAB17, morFenB21}. The greedy algorithm for steady systems and frequency-domain MOR has the same formulation, whereas POD-greedy for time-domain MOR of time-dependent systems needs an SVD step at each greedy iteration. In this work, we focus on the greedy algorithm, though the proposed scheme can be easily extended to POD-greedy. 
\begin{algorithm}
\caption{Standard greedy algorithm}
\label{alg:greedy}
\begin{algorithmic}[1]
\REQUIRE{the FOM, a training set $\Xi$ composed of parameter samples taken from the parameter domain $\mathcal P$, error tolerance tol$<1$,  $\Delta(\mu)$ to estimate the error.}
\ENSURE{Projection matrix $V$.}
\STATE Choose initial parameter $\mu^* \in \Xi$.
\STATE $V \gets \emptyset$, $\varepsilon=1$.
\WHILE{$\varepsilon>$tol}
\STATE Compute the snapshot(s) $x(\mu^*)$ by solving the FOM at $\mu=\mu^*$. 
\STATE Update $V$ by $V=\mathrm{orth}\{V, x(\mu^*)\}$,  (e.g., using the modified Gram-Schmidt process with deflation.)
\STATE Compute $\mu^*$ such that $\mu^*=\arg\max\limits_{\mu \in \Xi} \Delta(\mu)$. 
\STATE $\varepsilon=\Delta(\mu^*)$.
\ENDWHILE
\end{algorithmic}
\end{algorithm}
We consider constructing a ROM for the following full-order model (FOM) using the greedy algorithm,
\begin{equation}
F(x(\mu),\mu)=B(\mu).
\end{equation} 
Here, $F(x(\mu), \mu) \in \mathbb C^{n \times n_I}$, $x(\mu) \in \mathbb C^{n \times n_I}$, and $B(\mu) \in \mathbb C^{n\times n_I}$. $\mu \in \mathcal P$ is a parameter in the parameter domain $\mathcal P$. The variable $n$ is the order of the FOM, which can be the number of degrees of freedom after numerical discretization of PDEs describing a physical phenomenon. The proposed algorithms are also applicable to problems with $n_I>1$.

The ROM can be obtained via Galerkin projection using a projection matrix $V \in \mathbb R^{n\times r}$, $r\ll n$, as below,
\begin{equation}
\hat F(Vz(\mu),\mu)=\hat b(\mu),
\end{equation}
 where $\hat F(Vz(\mu),\mu)=V^Tf(Vz(\mu),\mu) \in \mathbb C^{r \times n_I}$, $z(\mu) \in \mathbb C^{r\times n_I}$, and $\hat B(\mu)=V^T b(\mu) \in \mathbb C^{r\times n_I}$. 
 
The standard greedy process used to compute the projection matrix $V$ is described in Algorithm~\ref{alg:greedy}.
Step 4 in Algorithm~\ref{alg:greedy} solves the FOM at $\mu^*$, and Step 6 computes an error estimator $\Delta(\mu)$ at all $\mu$ in $\Xi$. These two steps constitute the most computational expensive part of the greedy algorithm. However, Step 4 is unavoidable, since $x(\mu^*)$ is needed for the reduced basis construction.The computational complexity of Step 6 could be reduced, if the cardinality of $\Xi$, i.e., $|\Xi|$ is kept small, so that $\Delta(\mu)$ needs not be evaluated at many parameter samples. This is the motivation of the surrogate error estimator proposed in~\cite{morCheFB22, morCheFdetal21a}. We call the surrogate error estimator $\Delta_l(\mu)$ the low-fidelity error estimator as compared to the original error estimator $\Delta(\mu)$, since $\Delta_l(\mu)$ is only an approximation to $\Delta(\mu)$, but is much cheaper to compute.

In the next section, we present a greedy algorithm using bi-fidelity error estimation, where the low-fidelity error estimator is computed following the method in~\cite{morCheFB22, morCheFdetal21a}. Based on this, a greedy algorithm using multi-fidelity error estimation associated with a particular high-fidelity error estimator for MOR of linear parametric systems, is proposed. 

\section{Greedy algorithm with bi-fidelity and multi-fidelity error estimation}
\label{sec:multifidelity}
This section presents greedy algorithms with bi-fidelity and multi-fidelity error estimation, respectively. 

\subsection{Greedy algorithm with bi-fidelity error estimation}
\label{sec:bifidelity}
Algorithm~\ref{alg:greedy_bifidelity} is the greedy algorithm with bi-fidelity error estimation. Its original version using a different high-fidelity error estimator was firstly proposed in~\cite{morCheFB22}. 
\begin{algorithm}[t]
\caption{Greedy algorithm with bi-fidelity error estimation}
\label{alg:greedy_bifidelity}
\begin{algorithmic}[1]
\REQUIRE{the FOM, a training set $\Xi_c$ composed of a small number of parameter samples taken from the parameter domain $\mathcal P$, a set $\Xi_f$ composed a large number of parameter samples of $\mu$ from $\mathcal P$, error tolerance tol$<1$, $\Delta(\mu)$ to estimate the error.}
\ENSURE{Projection matrix $V$.}
\STATE Choose initial parameter $\mu^* \in \Xi_c$.
\STATE $V \gets \emptyset$, $\varepsilon=1$.
\WHILE{$\varepsilon>$tol}
\STATE Compute the snapshot(s) $x(\mu^*)$ by solving the FOM at $\mu=\mu^*$. 
\STATE Update $V$ by $V=\mathrm{orth}\{V, x(\mu^*)\}$,  (e.g., using the modified Gram-Schmidt process with deflation.)
\STATE Compute $\mu^*$ such that $\mu^*=\arg\max\limits_{\mu \in \Xi_c} \Delta(\mu)$. 
\STATE Compute $\mu^o$ such that $\mu^o=\arg\min\limits_{\mu \in \Xi_c} \Delta(\mu)$.  
\STATE Compute the low-fidelity error estimator $\Delta_l(\mu)$ using values of $\Delta(\mu)$ corresponding to the samples of $\mu$ in $\Xi_c$ via~(\ref{eq:RBF}) and (\ref{eq:RBF_weights}). 
\STATE Evaluate $\Delta_l(\mu)$ over $\Xi_f$ and pick out a parameter $\mu_c$ from the large parameter set $\Xi_f$ corresponding to the largest value of $\Delta_l(\mu)$, i.e., $\mu_c=\arg\max\limits_{\mu \in \Xi_f}$ from $\Xi_f$. 
\STATE Update the small parameter set $\Xi_c$: if $\Delta_l(\mu_c)>$tol, enrich $\Xi_c$ with $\mu_c$, i.e., $\Xi_c=\{\Xi_c, \mu_c\}$, if $\Delta(\mu^o)<$tol, remove $\mu^o$ from $\Xi_c$: $\Xi_c=\Xi_c \backslash \mu^o$.
\STATE $\varepsilon=\Delta(\mu^*)$.
\ENDWHILE
\end{algorithmic}
\end{algorithm}
The key step of Algorithm~\ref{alg:greedy_bifidelity} is Step 8, where the low-fidelity error estimator $\Delta_l(\mu)$ is computed using values of $\Delta(\mu)$ at the samples of $\mu$ in the small parameter set $\Xi_c$. Basically, $\Delta_l(\mu)$ is represented by a weighted sum of radial basis functions (RBFs), i.e., 
\begin{equation}
\label{eq:RBF}
\Delta_l(\mu)=\sum\limits_{i=1}^m w_i \Phi(\mu-\mu_i),
\end{equation}
where $\Phi(\mu)$ are RBFs, $\mu_i$ are the samples in $\Xi_c$, and $m$ is the cardinality of $\Xi_c$, which is small. Once $w_i$ are known, the low-fidelity error estimator $\Delta_l(\mu)$ is known. The weights $w_i$ are computed via enforcing $\Delta_l(\mu)$ to interpolate $\Delta(\mu)$ at $\mu_j, \forall \mu_j \in \Xi_c$, i.e., $\Delta_l(\mu_j)=\Delta(\mu_j)$. Inserting $\mu=\mu_j \in \Xi_c$ into~(\ref{eq:RBF}), the weights $w_i$ can be computed by solving the linear system as below,
\begin{equation}
\label{eq:RBF_weights}
\begin{array}{rcl}\left[
\begin{array}{cccc}
\Phi(\mu_1-\mu_1) & \ldots & \Phi(\mu_1-\mu_m)\\
\vdots & \vdots & \vdots \\
\Phi(\mu_m-\mu_1) & \ldots & \Phi(\mu_m-\mu_m)
\end{array}\right]
\left[
\begin{array}{c}
w_1\\
\vdots\\
w_m
\end{array}\right] & = &\left[
\begin{array}{c}
\Delta(\mu_1)\\
\vdots\\
\Delta(\mu_m)
\end{array}\right].
\end{array}
\end{equation}
Since values of $\Delta(\mu)$ at $\mu \in \Xi_c$ are available, the weights can be easily computed by solving the above small linear system with $m \times m$ being the dimension of the coefficient matrix. As this is a rather small system, we usually do not observe ill conditioning. Otherwise, we can use a regularized version of (4)~\cite{morCheFB22, morCheFdetal21a}.

At each iteration of the bi-fidelity greedy algorithm, the linear system is solved for once (Step 8), then the low-fidelity error estimator $\Delta_l(\mu)$ is evaluated over a larger parameter set $\Xi_f$ using the weighted sum in~(\ref{eq:RBF}) (Step 9). This process of computing the weights and evaluating $\Delta_l(\mu)$ is much faster than evaluating the high-fidelity error estimator over a training set $\Xi$ whose cardinality $|\Xi|$ is much larger than $|\Xi_c|$. This is usually the case for the standard greedy algorithm, where $|\Xi|>|\Xi_c|$. Finally, at each iteration of the bi-fidelity algorithm, the total computational cost of Steps 6-8: computing the high-fidelity error estimator $\Delta(\mu)$ over $\Xi_c$, solving the linear system~(\ref{eq:RBF_weights}) and evaluating the low-fidelity erorr estimator $\Delta_l(\mu)$ over $\Xi_f$ is still much cheaper
than computing the high-fidelity error estimator $\Delta(\mu)$ over a training set $\Xi$, whose cardinality is, e.g., twice that of $|\Xi_c|$, as shown in the numerical tests. 

Besides computing $\mu^*$ corresponding to the maximal value of the error estimator $\Delta(\mu)$ over $\Xi_c$, the minimal value of $\Delta(\mu)$ is also computed in Step 7. The corresponding parameter $\mu^o$ could be deleted from $\Xi_c$ if $\Delta(\mu^o)$ is already below the tolerance tol, see Step 10. In this way, the  cardinality of the training set $\Xi_c$ remains almost constant, and can further save computations as compared with enriching $\Xi_c$ only. We will show in the numerical results that adding and removing samples to and from $\Xi_c$ gets ROMs with similar accuracy (even smaller) as only adding samples to $\Xi_c$, but leads to even faster convergence of the greedy algorithm. 

\begin{remark}
\label{remark:bifidelity}
In Step 9, it is also possible to choose more than one parameter from $\Xi_f$ by modifying Step 9 as: choose $n_\textrm{add}$ samples from $\Xi_f$ corresponding to $n_\textrm{add}$ largest values of $\Delta_l(\mu)$. Similarly, In Step 7, one can also choose $n_\textrm{del}>1$ parameter samples corresponding to $n_\textrm{del}$ smallest values of $\Delta(\mu)$ from $\Xi_c$. However, this will more or less increase the computational time at each iteration, since more computations are needed to choose those samples. Furthermore, to make sure that only samples at which 
$\Delta_l(\mu)$ is larger than the tolerance tol are added to $\Xi_c$, and only samples at which $\Delta(\mu)$ is smaller than tol are removed, additional calculations are necessary to check if all the selected samples meet the above criteria and should be finally selected or removed (see Step 10). Therefore, adding/removing at most one parameter sample each time should be more efficient. In the numerical tests, we also show results when 
$n_\textrm{add}=n_\textrm{del}=2$ and $n_\textrm{add}=n_\textrm{del}=5$ at each iteration of Algorithm~\ref{alg:greedy_bifidelity}. 
\end{remark}

The bi-fidelity error estimation is general and can be applied to any high-fidelity error estimators. For example, the high-fidelity error estimator in~\cite{morCheFB22} estimates the error of the ROM for nonlinear time-dependent parametric systems in the time domain, while the high-fidelity error estimator in~\cite{morCheFdetal21a} estimates the error of the ROM in the frequency-domain for linear parametric systems. 

\subsection{Greedy algorithm with multi-fidelity error estimation}
The multi-fidelity error estimation we are going to introduce depends on the formulation of the high-fidelity error estimator $\Delta(\mu)$. To illustrate the basic concept, we use an error estimator proposed in~\cite{morFenB21} as the high-fidelity error estimator and discuss how to further reduce the computational load by using multi-fidelity error estimation. 

\subsubsection{An error estimator for linear parametric systems}
The error estimator is applicable to estimating the output error of the ROM for FOMs in the following linear parametric form,
\begin{equation}
\label{eq:steady}
\begin{array}{rcl}
M(\mu) x(\mu)&=&B(\mu),\\
y(\mu)&=&C(\mu) x(\mu).
\end{array}
\end{equation}
Here, $M(\mu)\in \mathbb R^{n\times n}$, $B(\mu) \in \mathbb R^{n\times n_I}$, $C(\mu) \in \mathbb R^{n_O\times n}$ , $x(\mu) \in \mathbb R^n$, $y(\mu) \in \mathbb R^{n_O \times n_I}$. We consider the general case that both $B(\mu)$ and $C(\mu)$ are matrices, i.e. systems with multiple inputs and multiple outputs. The ROM of the above linear parametric system can be derived via Galerkin projection using a projection matrix $V$ composed of  the reduced basis. That is,
\begin{equation}
\label{eq:steady_ROM}
\begin{array}{rcl}
\hat M(\mu) z(\mu)&=&\hat B(\mu),\\
\hat y(\mu)&=&\hat C(\mu) z(\mu),
\end{array}
\end{equation}
where $\hat M(\mu)=V^TM(\mu)V$, $\hat B(\mu)=V^TB(\mu)$, $\hat C(\mu)=C(\mu)V$.

For the general situation when both $B(\mu)$ and $C(\mu)$ are matrices, 
the error of the $i,j$-th entry of the output matrix $\hat y(\mu)$ is 
\begin{equation}
\label{eq:error}
\begin{array}{l}
|y_{ij}(\mu)-\hat y_{ij}(\mu)|\\
=|C_i(\mu)(M^{-1}(\mu)B(\mu)-V\hat {M}^{-1}(\mu)\hat B_j(\mu))| \\
= |C_i(\mu)M^{-1}(\mu)(B_j(\mu)-M(\mu)\underbrace{V \hat {M}^{-1}(\mu)\hat B_j(\mu))}_{\hat x_j(\mu):=Vz_{j}(\mu)|} \\
=|C_i(\mu)M^{-1}(\mu)r_j(\mu)|,
\end{array}
\end{equation}
where $C_i(\mu)$ is the $i$-th row of $C(\mu)$ and $B_j(\mu)$ is the $j$-th column of $B(\mu)$. Here, we have defined: $z_{j}(\mu)=\hat {M}^{-1}(\mu)\hat B_j(\mu)$, i.e., $\hat {M}(\mu) z_j(\mu)=\hat B_j(\mu)$, $\hat x_j (\mu):=Vz_j (\mu)$ and $r_j (\mu):=B_j(\mu)-M(\mu)\hat x_j (\mu)$.
It is clear that
$$\hat {M}(\mu) z_j(\mu)=\hat B_j(\mu)$$ 
is a reduced-order model of
\begin{equation}
\label{eq:primal} 
M(\mu) x_j(\mu)=B_j(\mu),
\end{equation}
and $\hat x_j(\mu)\approx x_j (\mu)$, the $j$-th column of $x(\mu)$. Finally, $r_j(\mu)$ is the residual induced by $\hat x_j(\mu)$.

From the last equation in~(\ref{eq:error}), it is clear that to compute the absolute error of $\hat y_{ij}$, we need to solve a residual system: 
\begin{equation}
\label{eq:residual}
M(\mu) x_{r_j}(\mu)=r_j (\mu).
\end{equation}
Instead, we construct a ROM of it:
\begin{equation}
\label{eq:residual_ROM}
V_r^TM(\mu)V_r z_{r_j}(\mu)=V_r^Tr_j(\mu),
\end{equation} 
so that $x_{r_j}(\mu) \approx \hat x_{r_j}(\mu)=V_rz_{r_j}(\mu)$ . Finally,
$$|y_{ij}(\mu)-\hat y_{ij}(\mu)|\approx |C_i(\mu) \hat x_{r_j}(\mu)|.$$
Note that $\hat x_{r_j}(\mu)$ depends on $B_j(\mu)$, since $r_j(\mu)$ depends on $B_j(\mu)$.  Each column $B_j(\mu)$ is associated with a $\hat x_{r_j}(\mu)$. The overall error of $\hat y(\mu)$ as a matrix can be estimated as:
\begin{equation}
\label{eq:estimator}
\|y(\mu)-\hat y(\mu)\|_{\mathrm{max}}:=\max\limits_{i,j} |y_{ij}(\mu)-\hat y_{ij}(\mu)|\approx \max\limits_{i,j} |C_i(\mu) \hat x_{r_j}(\mu)|=:\tilde \Delta(\mu).
\end{equation}

$\tilde \Delta(\mu)$ defined in~(\ref{eq:estimator}) is one of the error estimators proposed in~\cite{morFenB21}, where the proposed error estimators were shown to outperform other existing error estimators in the literature~\cite{morSmeZP19, morFenAB17} in terms of both accuracy and computational efficiency. Furthermore, it has been discussed in~\cite{morFenB21} that $\tilde \Delta(\mu)$ is even more accurate but has less computational complexity than other proposed estimators, including the one used in~\cite{morCheFdetal21a}. Even with this error estimator, the greedy algorithm could take several hours to converge for some complex systems, for example, the time-delay systems we consider in this work. For such systems, although the standard greedy algorithm can already be accelerated by the bi-fidelity greedy algorithm, we suggest a possibility to further improve the bi-fidelity greedy algorithm by introducing multi-fidelity error estimation. 

We notice that in order to compute $\tilde \Delta(\mu)$, an extra projection matrix $V_r$ has to be constructed for $\hat x_{r_j}(\mu)$. Although $\hat x_{r_j}(\mu)$ is dependent on the individual column of $B(\mu)$, the matrix $V_r$ can be uniformly constructed based on the whole matrix $B(\mu)$. Then $V_r$ is valid for any column of $B(\mu)$. It is proved in~\cite{morFenB21} that taking $V_r=V$ leads to $\tilde \Delta(\mu)$ identically zero for all $\mu$. Therefore, $V_r$ should be additionally computed.

\subsubsection{Standard greedy algorithm using $\tilde \Delta(\mu)$}
For easy understanding of the multi-fidelity error estimation, 
we first present Algorithm~\ref{alg:greedy_delta1pr}, the standard greedy algorithm using $\tilde \Delta(\mu)$ in~(\ref{eq:estimator}) as the error estimator. There, some additional steps are added to compute $V_r$, see Step 5, Steps 7-8.  In Step 7 of Algorithm~\ref{alg:greedy_delta1pr}, $V_r$ is not only updated by $x(\mu^r)$, but also by $V$. This is due to the fact that the solution $x_{r_j}(\mu)$ to the residual system in~(\ref{eq:residual}) can be written as
\begin{equation}
\label{eq:xr}
\begin{array}{rcl}
x_{r_j}(\mu)&=&M(\mu)^{-1}r_j (\mu) \\
&=&M(\mu)^{-1}(B_j(\mu)-M(\mu)\hat x_j (\mu))\\
&=&M(\mu)^{-1}B_j(\mu)-V z_j(\mu)\\
&\approx& \tilde V_r z_{r_j}-V z_j(\mu).
\end{array}
\end{equation}
It is clear that $x_{r_j}(\mu)$ is a linear combination of $(M(\mu))^{-1}B_j(\mu)$ and the columns of $V$. Therefore, $V$ contributes to the subspace approximating the solution space of $x_{r_j}(\mu)$ and cannot be neglected. It is also noticed that $(M(\mu))^{-1}B_j(\mu)$ is in fact the solution $x_j(\mu)$ in~(\ref{eq:primal}), while $Vz_j(\mu)$ is $\hat x_j(\mu)$ that approximates $x_j(\mu)$. This means $x_{r_j}(\mu)$ is the difference between $x_j(\mu)$ and $\hat x_j(\mu)$, which is a nonzero vector. Therefore, we should compute another matrix $\tilde V_r$, so that $x_j(\mu)\approx \tilde V_r z_{r_j}(\mu)$, but $\tilde V_r z_{r_j}(\mu)\neq \hat x_j(\mu)=V z_j(\mu)$. Finally, $x_{r_j}$ is approximated by the difference between $\tilde V_r z_{r_j}(\mu)$ and $V z_j(\mu)$. In other words, it is approximately represented as the linear combination of the columns of both $V_r$ and $V$. This approximation also explains Step 5 and Step 7 of Algorithm~\ref{alg:greedy_delta1pr}: Step 5 computes the reduced basis vectors contributing to $\tilde V_r$, Step 7 computes the complete reduced basis vectors contributing to $V_r$. New reduced basis vectors for both $V$ and $V_r$ are computed at each iteration of the greedy algorithm. Step 8 and Step 9 compute the new {\it important} parameter samples for $V_r$ and $V$, respectively. In general, $\mu^r$ should be different from $\mu^*$, since $\tilde \Delta(\mu) \neq \max\limits_{j=1,\ldots,n_I}\|r_j(\mu)-M(\mu)\hat x_{r_j}(\mu)\|$. Here, $r_j(\mu)-M(\mu)\hat x_{r_j}(\mu)$ is nothing but the residual induced by the approximate solution ($\hat x_{r_j}(\mu)$) to the residual system~(\ref{eq:residual}). 
\begin{algorithm}
\caption{Standard greedy algorithm using $\tilde \Delta(\mu)$ for linear parametric systems.}
\label{alg:greedy_delta1pr}
\begin{algorithmic}[1]
\REQUIRE{the FOM, a training set $\Xi$ composed of parameter samples taken from the parameter domain $\mu\in \mathcal P$, error tolerance tol$<1$.}
\ENSURE{Projection matrix $V$.}
\STATE Choose initial parameter $\mu^* \in \Xi$ for $V$, and initial parameter $\mu^r \neq \mu^*\in \Xi$ for $V_r$.
\STATE $V \gets \emptyset$, $V_r \gets \emptyset$, $\varepsilon=1$.
\WHILE{$\varepsilon>$tol}
\STATE Compute the snapshot(s) $x(\mu^*)$ by solving the FOM, i.e.  $x(\mu^*)=(M(\mu^*))^{-1} B(\mu^*)$.
\STATE Compute the snapshot(s) $x(\mu^r)$ by solving the FOM, i.e. $x(\mu^r)=(M(\mu^r))^{-1} B(\mu^r)$.
\STATE Update $V$ by $V=\mathrm{orth}\{V,x(\mu^*)\}$,  (e.g., using the modified Gram-Schmidt process with deflation.)

\STATE Update $V_r$ by $V_r=\mathrm{orth}\{V, V_r, x(\mu^r)\}$.

\STATE Compute $\mu^r$ such that $\mu^r=\arg\max\limits_{\mu \in \Xi} \max\limits_{j=1,\ldots,n_I}\|r_j(\mu)-M(\mu)\hat x_{r_j}(\mu)\|$, ($n_I$ is the total number of columns of $B(\mu)$).
\STATE Compute $\mu^*$ such that $\mu^*=\arg\max\limits_{\mu \in \Xi} \tilde \Delta(\mu)$. 
\STATE $\varepsilon=\tilde \Delta(\mu^*)$.
\ENDWHILE
\end{algorithmic}
\end{algorithm}
\subsubsection{Greedy algorithm with multi-fidelity error estimation}

The computational complexity of Algorithm~\ref{alg:greedy_delta1pr} using the error estimator $\tilde \Delta(\mu)$ comes from Steps 4-9. Efficiency of Step 9 can be improved by using the bi-fidelity error estimation as shown in Algorithm~\ref{alg:greedy_bifidelity}. The computations in Step 4, 6 are unavoidable, since $V$ is used to compute the ROM of the original FOM and should be updated till acceptable error tolerance is satisfied. In contrast, $V_r$ in Step 7 needs not be updated at every iteration. This implies that the ROM of the residual system does not have to be very accurate, since it is not the ROM that we seek, but an auxiliary ROM aiding the computation of $\tilde \Delta(\mu)$. 

An immediate consequence of Theorem 4.2 in~\cite{morFenB21} for single-input and single-output systems is the following Lemma for systems with multiple inputs and multiple outputs: 
\begin{lemma}
\label{theo:Ubound1pr}
The error of the output $\hat y(\mu)$ of the ROM~(\ref{eq:steady_ROM}) can be bounded as
\begin{equation}
\label{eq:bound1pr}
\tilde \Delta (\mu)-\delta(\mu)\leq \|y(\mu)-\hat y(\mu)\|_{\mathrm{max}} \leq \tilde \Delta(\mu)+\delta(\mu),
\end{equation}
where $\delta(\mu):=\max\limits_{i,j}|C_i(\mu)(x_{r_j}(\mu)-\hat x_{r_j}(\mu))|\geq 0$. 
\end{lemma} 
\begin{proof}
From~(\ref{eq:error}), we know $$|y_{ij}(\mu)-\hat y_{ij}(\mu)|=|C_i(\mu)x_{r_j}(\mu)|\approx |C_i(\mu) \hat x_{r_j}(\mu)|.$$
Then 
\begin{equation}
\label{eq:upb}
\begin{array}{rcl}
|y_{ij}(\mu)-\hat y_{ij}(\mu)|&=&|C_i(\mu)x_{r_j}(\mu)|+|C_i(\mu)\hat x_{r_j}(\mu)|-|C_i(\mu)\hat x_{r_j}(\mu)|\\
&\leq& |C_i(\mu)\hat x_{r_j}(\mu)|+\underbrace{|C_i(\mu)x_{r_j}(\mu)-C_i(\mu)\hat x_{r_j}(\mu)|}_{\delta_{ij}(\mu)}.
\end{array}
\end{equation}
On the other hand, 
\begin{equation}
\label{eq:lowb}
\begin{array}{rcl}
|C_i(\mu)\hat x_{r_j}(\mu)|&=&|C_i(\mu)\hat x_{r_j}(\mu)|+|C_i(\mu) x_{r_j}(\mu)|-|C_i(\mu)x_{r_j}(\mu)|\\
&\leq&|C_i(\mu)x_{r_j}(\mu)|+\delta_{ij}(\mu).
\end{array}
\end{equation}
From~(\ref{eq:estimator}), (\ref{eq:lowb}) and the definition of $\delta(\mu)$, we have $$\tilde \Delta(\mu)=\max\limits_{i,j}|C_i(\mu)\hat x_{r_j}(\mu)| \leq \|y(\mu)-\hat y(\mu)\|_{\mathrm{max}}+\delta(\mu).$$
Similarly, from~(\ref{eq:upb}), we get
$$\|y(\mu)-\hat y(\mu)\|_{\mathrm{max}} \leq \tilde \Delta(\mu)+\delta(\mu).$$
This completes the proof. \qed
\end{proof}
From the definition of $\delta(\mu)$, it is seen that the more accurate the ROM of the residual system, the smaller $\delta(\mu)$ is. As a result, $\tilde \Delta(\mu)$ should measure the true error more accurately so that the {\it important} parameters it selects are closer to those selected by the true error, given the same training set $\Xi$. On the contrary, if the ROM of the residual system is less accurate, $\tilde \Delta(\mu)$ will be less accurate, too. However, at a certain stage, when $\tilde \Delta(\mu)$ is already small, the right-hand side of the residual system $r_j(\mu)$ will also be small, so that it can be expected that both $x_{r_j}(\mu)$ and $\hat x_{r_j}(\tilde \mu)$ are close to zero. This leads to a small $\delta(\mu)$. Variation of a small $\delta(\mu)$ will not cause big variation of the difference between $\tilde \Delta(\mu)$ and the true error $\|y(\mu)-\hat y(\mu)\|_{\mathrm{max}}$. The trend, though not the exact route, of error decay could still be anticipated so that {\it important} parameters corresponding to the error peaks can also be detected. The above analyses are also justified by the numerical results in the next section, see, e.g., Figure~\ref{fig:divider2} and Figure~\ref{fig:coplanar}. 

This motivates the multi-fidelity error estimation. We set a second tolerance $\epsilon>$tol, and when $\tilde \Delta(\mu)<\epsilon<1$, we stop updating the ROM of the residual system, i.e., stop implementing Step 5, Step 7 and Step 8 of Algorithm~\ref{alg:greedy_delta1pr}. The error estimator $\tilde \Delta(\mu)$ after this stage may not be as accurate as it would be when keep updating the ROM of the residual system. However, the difference should be small as $\tilde \Delta(\mu)$ is already below a small value $\epsilon$. Without implementing Step 5,  we have saved computations of simulating the FOM. For large and complex systems, solving the FOM even once is not fast. The computation in Step 7 is relatively cheap if the system is not very large. The computational cost in Step 8 is not low for certain complex problems, though some $\mu$-independent parts of $r_j(\mu)$ and $M(\mu)$ can be pre-computed. For example, this is the case for the time-delay systems in the next section. 

Stop updating the ROM of the residual system gives rise to a low-fidelity error estimator at later iteration steps of the greedy algorithm. When this low-fidelity error estimator is combined with $\Delta_l(\mu)$ in Algorithm~\ref{alg:greedy_bifidelity}, we obtain the multi-fidelity error estimation. This is detailed in Algorithm~\ref{alg:greedy_multifidelity}. Compared with the standard greedy algorithm, the overall saving in computational costs is noticeable, which can be seen from the numerical results in the next section.

The concept of multi-fidelity error estimation could also be applied to other high-fidelity error estimators. For example, Step 15 could be modified as ``Stop updating partial information of $\Delta(\mu)$", if some parts of the high-fidelity error estimator $\Delta(\mu)$ are not ``essential" for computing $\Delta(\mu)$. 
\begin{algorithm}[h]
\caption{Greedy algorithm with multi-fidelity error estimation}
\label{alg:greedy_multifidelity}
\begin{algorithmic}[1]
\REQUIRE{the FOM, a training set $\Xi_c$ composed of a small number of parameter samples taken from the parameter domain $\mu\in \mathcal P$, a set $\Xi_f$ composed of a large number of parameter samples of $\mu$ from $\mathcal P$, error tolerance tol$<1$.}
\ENSURE{Projection matrix $V$.}
\STATE Choose initial parameter $\mu^* \in \Xi_c$ for $V$, and initial parameter $\mu^r \neq \mu^*\in \Xi_c$ for $V_r$.
\STATE $V \gets \emptyset$, $V_r \gets \emptyset$, $\varepsilon=1$.
\WHILE{$\varepsilon>$tol}
\STATE Compute the snapshot(s) $x(\mu^*)$ by solving the FOM,  i.e. $x(\mu^*)=(M(\mu^*))^{-1} B(\mu^*)$. 
\STATE Compute the snapshot(s) $x(\mu^r)$ by solving the FOM, i.e. $x(\mu^r)=(M(\mu^r))^{-1} B(\mu^r)$.
\STATE Update $V$ by $V=\mathrm{orth}\{V,x(\mu^*)\}$ (e.g., using the modified Gram-Schmidt process with deflation).
\STATE Update $V_r$ by $V_r=\mathrm{orth}\{V, V_r, x(\mu^r)\}$.
\STATE Compute $\mu^*$ such that $\mu^*=\arg\max\limits_{\mu \in \Xi_c} \tilde \Delta(\mu)$. 
\STATE Compute $\mu^o$ such that $\mu^o=\arg\min\limits_{\mu \in \Xi_c} \tilde \Delta(\mu)$. 
\STATE Compute $\mu^r$ such that $\mu^r=\arg\max\limits_{\mu \in \Xi_c}\max\limits_{j=1,\ldots,n_I} \|r_j(\mu)-M(\mu)\hat x_{r_j}(\mu)\|$,  \hfill \% $n_I$ is the total number of columns of $B(\mu)$.
\STATE Compute the low-fidelity error estimator $\tilde \Delta_l(\mu)$ using values of $\tilde \Delta(\mu)$ corresponding to the samples of $\mu$ in $\Xi_c$ via~(\ref{eq:RBF}) and (\ref{eq:RBF_weights}). 
\STATE Evaluate $\tilde \Delta_l(\mu)$ over $\Xi_f$ and pick out a parameter $\mu_c$ from the large parameter set $\Xi_f$ corresponding to the largest value of $\tilde \Delta_l(\mu)$, i.e., $\mu_c=\arg\max\limits_{\mu \in \Xi_f}$ from $\Xi_f$. 
\STATE Update the small parameter set $\Xi_c$: if $\Delta_l(\mu_c)>$tol, enrich $\Xi_c$ with $\mu_c$, i.e., $\Xi_c=\{\Xi_c, \mu_c\}$, if $\Delta(\mu^o)<$tol, remove $\mu^o$ from $\Xi_c$: $\Xi_c=\Xi_c \backslash \mu^o$.
\STATE $\varepsilon=\tilde \Delta(\mu^*)$.
\IF {$\varepsilon < \epsilon$}
\STATE Stop performing Step 5, Step 7 and Step 10.   \hfill \% stop updating the ROM of the residual system.
\ENDIF
\ENDWHILE
\end{algorithmic}
\end{algorithm}

\subsection{Application to MOR for time-delay systems}
\label{sec:multifidelity_timedelay}
In this section, we consider applying Algorithm~\ref{alg:greedy_bifidelity}, the greedy algorithm with bi-fidelity error estimation, Algorithm~\ref{alg:greedy_delta1pr}, the standard greedy algorithm and Algorithm~\ref{alg:greedy_multifidelity}, the greedy algorithm with multi-fidelity error estimation to large-scale time-delay systems with many delays. The time-delay systems are defined as:
\begin{equation}
\label{eq:prelim:dde}
\begin{split}
\sum_{j=0}^d E_j \dot{x}(t - \tau_j) &= \sum_{j=0}^d A_j x(t - \tau_j) + B u(t), \\
y(t) &= C x(t),
\end{split}
\qquad \forall \ t \geq 0
\end{equation}
with an initial condition
$x(t) = \Phi(t) \in \mathbb C^n, \forall\ t \in [-\tau_d, 0]$.
Here,
$E_0, \ldots, E_d , A_0, \ldots, A_d \in \mathbb C^{n \times n}, B \in \mathbb C^{n \times n_I}, C \in \mathbb C^{n_O \times n},$ 
$0 = \tau_0 < \tau_1 < \ldots < \tau_d$ and $n$ is called the order of the delay system.
The transfer function of the delay system is defined as:
\begin{equation}
\label{eq:transfun}
H(s)=C {\mathcal  K}^{-1}(s)B,
\end{equation}
where ${\mathcal  K}(s) = s\sum_{j=0}^d E_j e^{-s\tau_j}- \sum_{j=0}^d A_j e^{-s\tau_j}$, $s=2\pi \jmath$ is the variable in the frequency domain, $f$ is the ordinary frequency with unit Hz and $\jmath$ is the imaginary unit. 

A ROM of the delay system, which has the same delays as the original system, can be obtained via Galerkin projection using a projection matrix $V\in \mathbb R^{n\times r}$, $r\ll n$, i.e.,
\begin{equation}
\label{eq:delay_ROM}
\begin{split}
\sum_{j=0}^d \hat E_j \dot{z}(t - \tau_j) &= \sum_{j=0}^d \hat A_j z(t - \tau_j) + \hat B u(t), \\
\hat y(t) &= \hat C z(t),
\end{split}
\qquad \forall \ t \geq 0,
\end{equation}
where $\hat E_j=V^TE_jV \in \mathbb R^{r\times r}, \hat A_j=V^TA_jV \in \mathbb R^ {r\times r}, \hat B=V^TB \in \mathbb R^{r\times n_I}, \hat C=CV \in \mathbb R^{n_O\times r}$, with $r \ll n$ being the order of the ROM. The original state vector $x(t)$
in~(\ref{eq:prelim:dde}) can be recovered by the approximation: $x(t)\approx Vz(t)$. The transfer function of the ROM is
\begin{equation*}
\label{eq:transfun_ROM}
\hat H(s)=\hat C \hat{\mathcal K}^{-1}(s)\hat B,
\end{equation*}
where 
$\hat {\mathcal  K}(s)=s \sum_{j=0}^d \hat E_j e^{-s\tau_j}- \sum_{j=0}^d \hat A_j e^{-s\tau_j}$.
The projection matrix $V$ can be constructed via approximating $H(s)$~\cite{morBeaG09} as follows. Note that $H(s)$ is nothing but the output $y( \mu)$ of the linear parametric system in~(\ref{eq:steady}), with $M(\mu)=\mathcal K(s)$, $B(\mu)=B$ and $\mu=s$, i.e.,
\begin{equation}
\label{eq:delay_freq}
\begin{array}{rcl}
\mathcal K(s)x(s)&=&B,\\
H(s)&=&C(s)x(s).
\end{array}
\end{equation}
The reduced transfer function $\hat H(s)$ is the output $\hat y(\mu)$ of the ROM in~(\ref{eq:steady_ROM}) with $\hat M(\mu)=\hat K(s)$ and $\hat B(\mu)=\hat B$. 

It is easy to see that the projection matrix $V$ that is used to construct the ROM~(\ref{eq:delay_ROM}) in the time domain is exactly the same matrix to obtain the reduced transfer function $\hat H(s)$. Therefore, $V$ can be obtained by constructing a ROM of system~(\ref{eq:delay_freq}) in the frequency domain, i.e., by approximating the transfer function $H(s)$. This can be done by the standard greedy Algorithm~\ref{alg:greedy_delta1pr} with the error estimator $\tilde \Delta(s)$, where $V$ is iteratively computed by choosing proper samples of $s$~\cite{morBeaG09, morAlfFLetal21}. In fact, the reduced transfer function $\hat H(s)$ interpolates the original transfer function $H(s)$ at the selected samples of $s$~\cite{morAlfFLetal21}. The matrix $M(\mu)$ in Steps 4-5 of Algorithm~\ref{alg:greedy_delta1pr} is now replaced by $\mathcal K(s)$. The difference of the coefficient matrix $\mathcal K(s)$ from a single matrix $M(\mu)$ in the usual case is its high complexity. To solve the system in~(\ref{eq:delay_freq}) is much more expensive than solving the system in~(\ref{eq:steady}) where $M(\mu)$ is a single matrix. On the one hand, the matrices constituting $\mathcal K(s)$ must be assembled to get $\mathcal K(s)$. On the other hand, the finally assembled matrix has some dense blocks, though each single matrix contributing to $\mathcal K(s)$ is sparse. 

To further improve the efficiency of the standard greedy algorithm, we propose to apply Algorithm~\ref{alg:greedy_bifidelity} and Algorithm~\ref{alg:greedy_multifidelity} to time-delay systems. The application is straightforward by simply replacing the FOM in~(\ref{eq:steady}) in both algorithms with the system in~(\ref{eq:delay_freq}), i.e., the matrix $M(\mu)$ is replaced by $\mathcal K(s)$, the input matrix $B(\mu)$ and the output matrix $C(\mu)$ are replaced by $B$ and $C$ in~(\ref{eq:delay_freq}), respectively. 

\section{Numerical tests}
\label{sec:results}
We consider three time-delay systems obtained from partial element equivalent circuit (PEEC) modelling and simulation, which transfer problems from the electromagnetic domain to the circuit domain~\cite{Rueh72,Rueh73,Rueh74,Ru17AJ}. When the
propagation delays are explicitly kept for both partial inductances and coefficients of potential, time-delay systems can be derived~\cite{GiLoAn19}. Numerical tests are done with MATLAB R2016b on a computer server with 4 Intel Xeon E7-8837 CPUs running at 2.67 GHz, 1TB main memory, split into four 256 GB partitions.

We test the standard greedy Algorithm~\ref{alg:greedy_delta1pr}, the bi-fidelity greedy Algorithm~\ref{alg:greedy_bifidelity} and the multi-fidelity Algorithm~\ref{alg:greedy_multifidelity} on three time-delay systems. To run the algorithms, we need to initialize the algorithms by doing the following:
\begin{itemize}
\item The samples in the training set $\Xi$, the small set $\Xi_c$ and the large set $\Xi_f$ are taken from the prescribed frequency domain and are generated using the MATLAB function \texttt{linspace}: $\texttt{linspace}(f_l, f_h, \mathrm{cardi})$. Here, $f_l$ is the lowest frequency, $f_h$ is the highest frequency used in \texttt{linspace}, $\mathrm{cardi}$ is the corresponding cardinality of each set. The samples of $s$ are then computed using the relation: $s=2 \pi \jmath f$. 

\item For the multi-fidelity error estimation, we set $\epsilon=0.1$ in Step 15 of Algorithm~\ref{alg:greedy_multifidelity}. 

\item To compute the low-fidelity error estimator, we choose the inverse multiquadratic RBF (IMQ) $\Phi=\frac{1}{1+(a \|\mu-\mu_i\|)^2}$ with the shape parameter $a=30$.
\end{itemize}
We also need to define some variables uniformly used in all the tables and figures:
\begin{itemize}
\item The error $\|H(s)-\hat H(s)\|_{\textrm{max}}$ of the transfer function $\hat H(s)$ of the ROM is finally computed over 1000 samples of $s$ drawn independently of the training sets, resulting in the validated error: Valid.err in Tables~\ref{tab:divider1}-\ref{tab:ustrip2}. 

\item Runtime, the walltime of each algorithm till convergence.

\item Iter., the total number of iterations of each algorithm.

\item $r$, the order of the ROM.

\item The high-fidelity error estimator at each iteration of Algorithm~\ref{alg:greedy_delta1pr} is defined as $\max\limits_{\mu \in \Xi} \tilde \Delta(\mu)$.

\item The bi-fidelity error estimator at each iteration of Algorithm~\ref{alg:greedy_bifidelity} is defined as $\max\limits_{\mu \in \Xi_c} \tilde \Delta(\mu)$.

\item The multi-fidelity error estimator at each iteration of Algorithm~\ref{alg:greedy_multifidelity} is defined as $\max\limits_{\mu \in \Xi_c} \tilde \Delta(\mu)$. Here $\tilde \Delta(\mu)$ will be different from the bi-fidelity error estimator once Step 15 of the algorithm takes action.

\item The true error at each iteration of Algorithm~\ref{alg:greedy_delta1pr} is defined as $\max\limits_{\mu \in \Xi} \|H(s)-\hat H(s)\|_{\textrm{max}}$.

\item The true error at each iteration of Algorithm~\ref{alg:greedy_bifidelity} or Algorithm~\ref{alg:greedy_multifidelity} is defined as $\max\limits_{\mu \in \Xi_c} \|H(s)-\hat H(s)\|_{\textrm{max}}$.
\end{itemize}
Note that $\Xi_c$ could be enriched only by adding samples from $\Xi_f$ to $\Xi_c$. 
As the high-fidelity error estimator $\tilde \Delta(\mu)$ needs to be computed at every sample in $\Xi_c$ at each iteration, samples in $\Xi_c$ whose corresponding error is already smaller than tol can also be removed from $\Xi_c$ to keep the cardinality of $\Xi_c$ constant, so that more computations can be saved. We consider both cases separately and compare their efficiency with respect to both runtime and accuracy.
\subsection{Test 1: results for a model of three-port divider}
\label{subsec:model1}
The model structure of a three-port microstrip power-divider circuit is shown in Fig. \ref{fig:ThreePort} ($P_1, P_2$ and $P_3$ denote the ports). The dimensions of the circuit are [20, 20, 0.5] mm in the $[x, y, z]$ directions and the width of the microstrips is set as 0.8 mm. Furthermore, the dimensions $l_{X1}$, $l_{Y1}$, and $l_{Y3}$ are $9$, $7.2$ and $7.2$ mm, respectively. The relative dielectric constant is $\varepsilon_r=2.2$. All the ports are terminated on 50 $\Omega$ resistances. The order of the FOM is $n=10,626$, and it has $d=93$ delays. The interesting frequency band is $[0,20]$GHz.
\begin{figure}
    \centering
    \includegraphics[width=8cm]{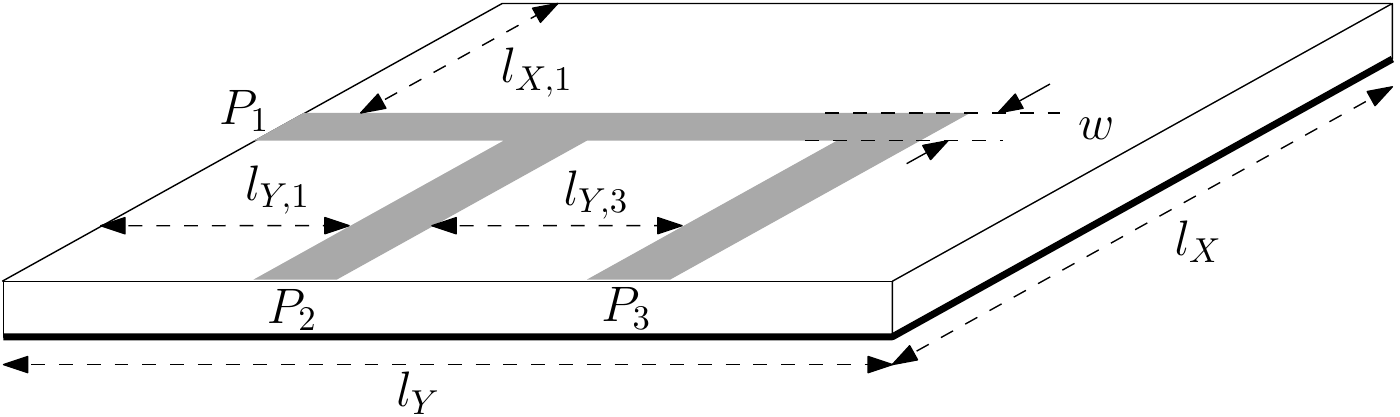}
    \caption{The three-port microstrip power-divider circuit.}
        \label{fig:ThreePort}
\end{figure}

For this model, we use $f_l=1\times10^6, f_h=2\times 10^{10}$ in the function \texttt{linspace}. $|\Xi|=30$ or $|\Xi|=40$ for the standard greedy Algorithm~\ref{alg:greedy_delta1pr}. For Algorithm~\ref{alg:greedy_bifidelity} and Algorithm~\ref{alg:greedy_multifidelity}, $|\Xi_c|=15$ or  $|\Xi_c|=20$ and $|\Xi_f|=100$. The set $\Xi_c$ is then updated during the iteration of the greedy algorithm. The 1000 samples for validating the ROM accuracy are created using the MATLAB function \texttt{logspace}, i.e., \texttt{logspace}$(\log_{10}(f_{l_1}), \log_{10}(f_h), 1000)$. $f_{l_1}=1\times 10^4$.

In Table~\ref{tab:divider1}, we list the results of the three algorithms. The standard greedy algorithm is the slowest. The other algorithms are all much faster and take at least 2 hours less  than the standard algorithm. The bi-fidelity greedy algorithm by enriching $\Xi_c$ only is slower than other bi-(multi-)fidelity algorithms, this is in agreement with our theoretical analysis in Section~\ref{sec:multifidelity}. The multi-fidelity algorithm by adding and removing samples to and from $\Xi_c$ performs the best in terms of runtime and accuracy. Compared to the standard algorithm, it has reduced the offline runtime from 5.6 hours to 1.8 hours, and almost 4 hours have been saved. Finally, a speed-up factor 3.1 is achieved. Except for the bi-fidelity algorithm by adding and removing samples, the other algorithms have produced ROMs with validated errors below the tolerance. The bi-fidelity algorithms perform similarly as the standard algorithm. All three algorithms converge in 14 iterations, and produce ROMs smaller than the others. 

It is worth pointing out that if using fewer samples in $\Xi$ for the standard greedy algorithm, the ROM has a validated error that is slightly larger than the tolerance, as shown in Table~\ref{tab:divider2}, where $|\Xi|=30$. Also, the bi-fidelity greedy algorithms are less accurate if using fewer samples in $\Xi_c$, as shown in Table~\ref{tab:divider2}. There, the same $\Xi_c$ used for the multi-fidelity greedy algorithms are used, but less accurate ROMs are obtained.   

%
%
\begin{table}[!h]
\centering
\caption{Three-port divider: $n=10,626$, $d=93$ delays, tol=0.001, adding/removing a single sample at each iteration.}
\label{tab:divider1}
\begin{tabular}{|l|l|l|l|l|}
\hline
Method & Iter. & Runtime (h) & $r$ & Valid.err  \\ \hline
Alg.~\ref{alg:greedy_delta1pr} (standard, $|\Xi|=40$) &14  & 5.6 & 84 & $9.2\times 10^{-4}$ \\ \hline
Alg.~\ref{alg:greedy_bifidelity} (bi-fidelity, add only, $|\Xi_c|=20$) &14  &3.6 & 84 &$6\times 10^{-4}$   \\ \hline
Alg.~\ref{alg:greedy_bifidelity} (bi-fidelity, add-remove, $|\Xi_c|=20$)& 14 & 2.7 &84 &0.0022  \\ \hline
Alg.~\ref{alg:greedy_multifidelity} (multi-fidelity, add only, $|\Xi_c|=15$) &15  &2.4  &90 &$6.2\times 10^{-4}$   \\ \hline
Alg.~\ref{alg:greedy_multifidelity} (multi-fidelity, add-remove, $|\Xi_c|=15$)&15 & 1.8 & 90&$6.2\times 10^{-4}$  \\ \hline
\end{tabular}
\end{table}
\begin{table}[!h]
\centering
\caption{Three-port divider: $n=10,626$, $d=93$ delays, tol=0.001, smaller $|\Xi|$ and $|\Xi_c|$, adding/removing a single sample at each iteration.}
\label{tab:divider2}
\begin{tabular}{|l|l|l|l|l|}
\hline
Method & Iter. & Runtime (h) & $r$ & Valid.err  \\ \hline
Alg.~\ref{alg:greedy_delta1pr} (standard, $|\Xi|=30$) &14  & 4.2& 84&  0.0017       \\ \hline
Alg.~\ref{alg:greedy_bifidelity} (bi-fidelity, add only,$|\Xi_c|=15$) &13  &2.5 & 78 &0.0026   \\ \hline
Alg.~\ref{alg:greedy_bifidelity} (bi-fidelity, add-remove, $|\Xi_c|=15$)& 13 & 1.9 &78 &0.0088  \\ \hline
\end{tabular}
\end{table}
\begin{table}[!h]
\centering
\caption{Three-port divider: $n=10,626$, $d=93$ delays, tol=0.001, adding/removing $n_\textrm{add}=n_\textrm{del}>1$ samples at each iteration.}
\label{tab:divider3}
\begin{tabular}{|l|l|l|l|l|}
\hline
Method & Iter.& Runtime (h) & $r$ & Valid.err  \\ \hline
Alg.~\ref{alg:greedy_bifidelity} (bi-fidelity, add-remove, $|\Xi_c|=15$, $n_\textrm{add}=2$)& 14 & 2.0 &84 &0.0022  \\ \hline
Alg.~\ref{alg:greedy_bifidelity} (bi-fidelity, add-remove, $|\Xi_c|=20$, $n_\textrm{add}=2$)& 14 & 2.7 &84 &0.0022  \\ \hline
Alg.~\ref{alg:greedy_bifidelity} (bi-fidelity, add-remove, $|\Xi_c|=20$, $n_\textrm{add}=5$)& 14 & 2.7 &84 &0.0022  \\ \hline
Alg.~\ref{alg:greedy_multifidelity} (multi-fidelity, add-remove, $|\Xi_c|=15$, $n_\textrm{add}=2$) &14  &1.7  &84 &0.0039   \\ \hline
Alg.~\ref{alg:greedy_multifidelity} (multi-fidelity, add-remove, $|\Xi_c|=15$, $n_\textrm{add}=5$)&14 & 1.7 & 84&0.0039 \\ \hline
\end{tabular}
\end{table}

In Table~\ref{tab:divider3}, we show the results of the bi-fidelity greedy algorithm and the multi-fidelity greedy algorithm when $n_\textrm{add}=n_\textrm{del}>1$  samples are added or removed from the small training set $\Xi_c$ at each iteration of the algorithm. In general, they produce similar results as those in Table~\ref{tab:divider1} and Table~\ref{tab:divider2} given the same $\Xi_c$. For $|\Xi_c|=15$, the bi-fidelity greedy algorithm with $n_\textrm{add}=n_\textrm{del}=2$  converges in 14 iterations, running one more iteration than with $n_\textrm{add}=n_\textrm{del}=1$ as shown in Table~\ref{tab:divider2}, and generates a ROM with slightly higher accuracy. On the contrary, given $|\Xi_c|=15$, the multi-fidelity greedy algorithm with either $n_\textrm{add}=n_\textrm{del}=2$ or $n_\textrm{add}=n_\textrm{del}=5$ runs one iteration less than in the case of adding/removing a single sample as shown in Table~\ref{tab:divider1}, and constructs ROMs with lower accuracy. Furthermore, it is seen that increasing $n_\textrm{add}=n_\textrm{del}$ from 2 to 5 did not change the results for both algorithms. In general, adding/removing a single sample keeps the algorithms simple but efficient. 

To illustrate the behavior of the error estimators further, we plot the decay of error estimators and their corresponding true errors during the greedy iterations. Since different $\mu^*$ are chosen according to different error estimators, the projection matrix $V$ is updated with different snapshots, leading to ROMs with different accuracy. Consequently, the true errors of the ROMs are expected to be different. 

Figures~\ref{fig:divider1}-\ref{fig:divider2} are the results of the algorithms in Table~\ref{tab:divider1}.
The left part of Figure~\ref{fig:divider1} shows the error of the high-fidelity error estimator at each iteration of Algorithm~\ref{alg:greedy_delta1pr} and the decay of the corresponding true error. The error estimator almost exactly matches the true error at all the iterations. The right part of Figure~\ref{fig:divider1} plots the decay of the bi-fidelity error estimator with respect to the true error. The bi-fidelity error estimator in both of the two cases: only adding (add-only) samples to $\Xi_c$, adding and removing (add-remove) samples to and from $\Xi_c$, can accurately catch the true error. Both cases converge in 14 iterations, but the case ``add-only" is more accurate as can be seen from Table~\ref{tab:divider1}. 

Figure~\ref{fig:divider2} plots the decay of the multi-fidelity error estimator and the corresponding true error decay. For clarity, the two cases ``add-only'' and ``add-remove'' are plotted in two separate figures. The multi-fidelity error estimator is not as accurate as the bi-fidelity error estimator. This is indicated by the error decay from the 10-th iteration to the end in both figures. From the 10-th iteration, the error estimator is below $\epsilon=0.1$, the multi-fidelity error estimation at Step 15 of Algorithm~\ref{alg:greedy_multifidelity} begins to be implemented. For this example, the multi-fidelity error estimator overestimates the true error more often than the bi-fidelity error estimator, it did not choose the interpolation points that lead to error decay as fast as those chosen by the bi-fidelity error estimator. Finally, it uses more iteration steps to converge. Whereas, they still produce ROMs with best accuracy. 
\begin{figure}[!h]
\centering
\includegraphics[width=75mm]{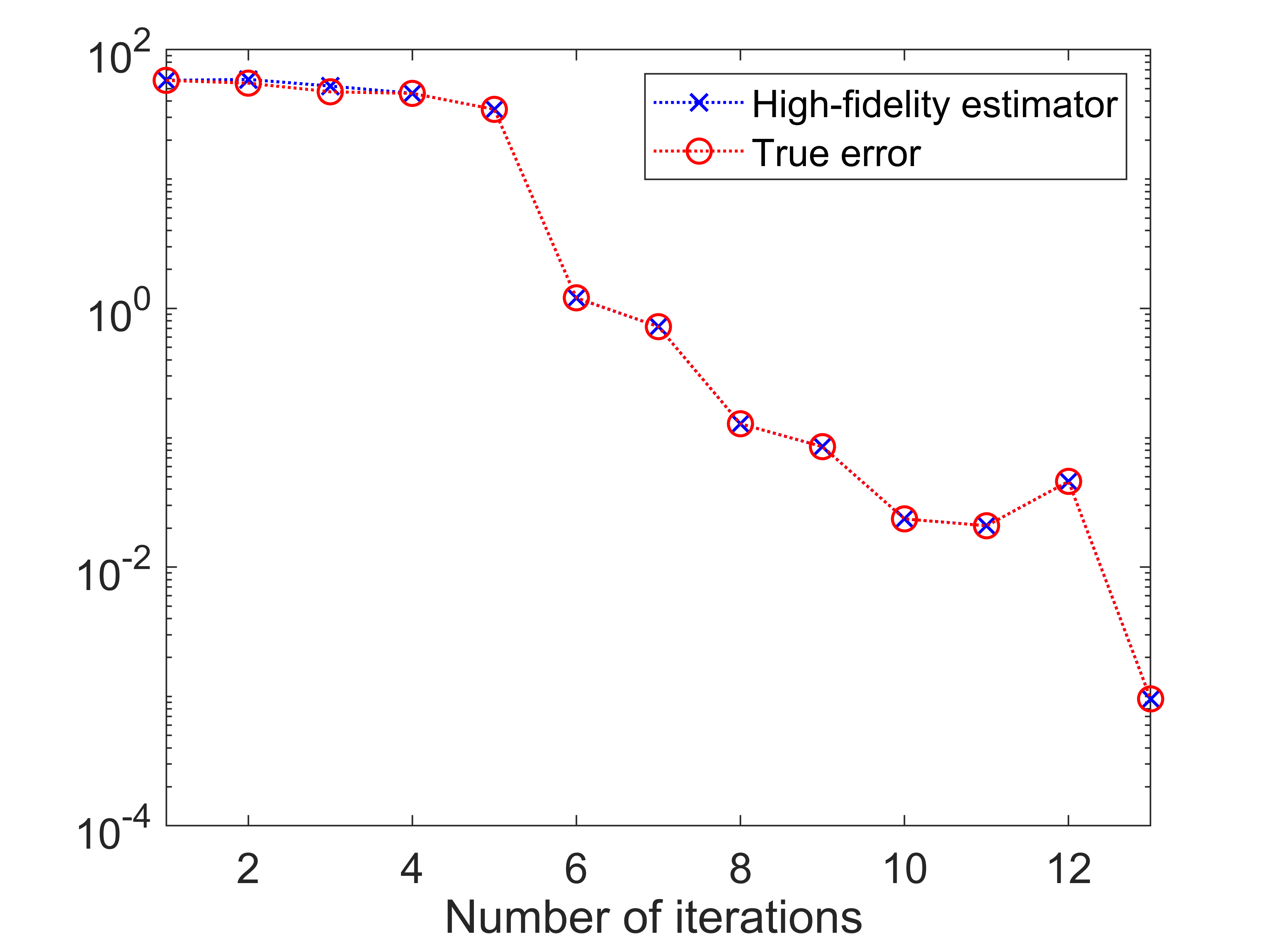}\qquad
\includegraphics[width=75mm]{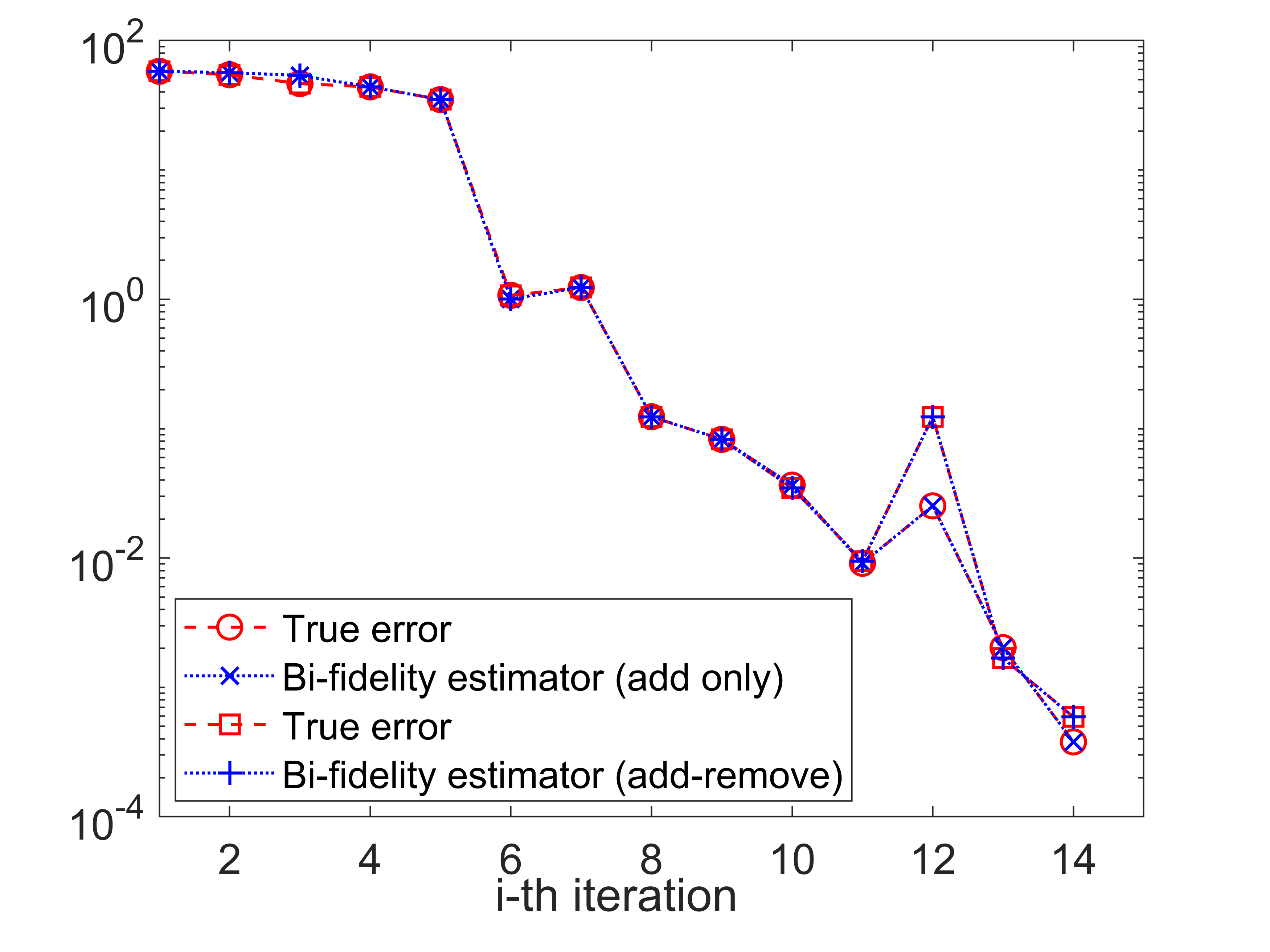}
\caption{Error decay. Left: true error vs high-fidelity error estimator. Right: true error vs bi-fidelity error estimators.}
\label{fig:divider1}
\end{figure}
\begin{figure}[!h]
\centering
\includegraphics[width=75mm]{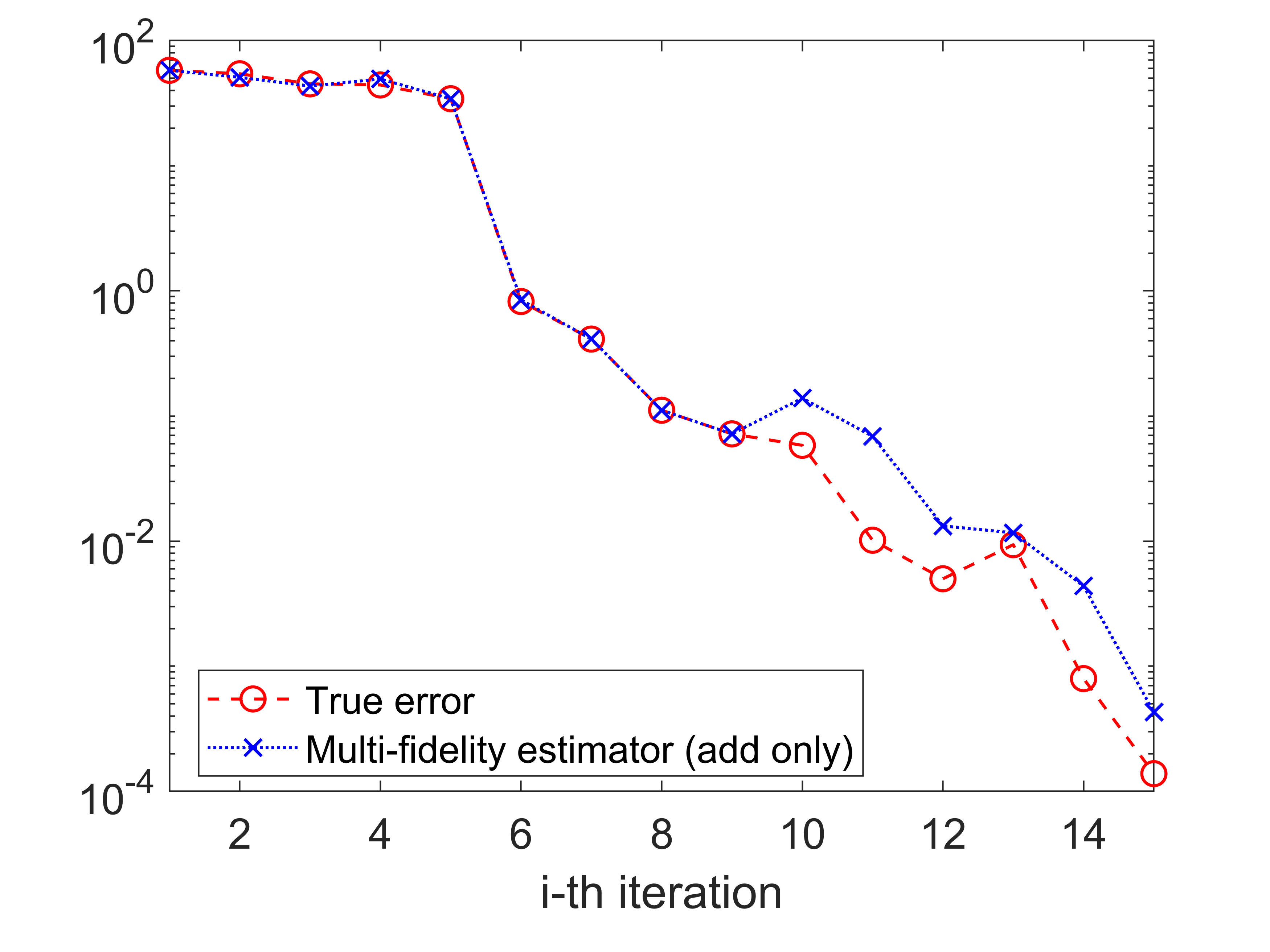}\qquad
\includegraphics[width=75mm]{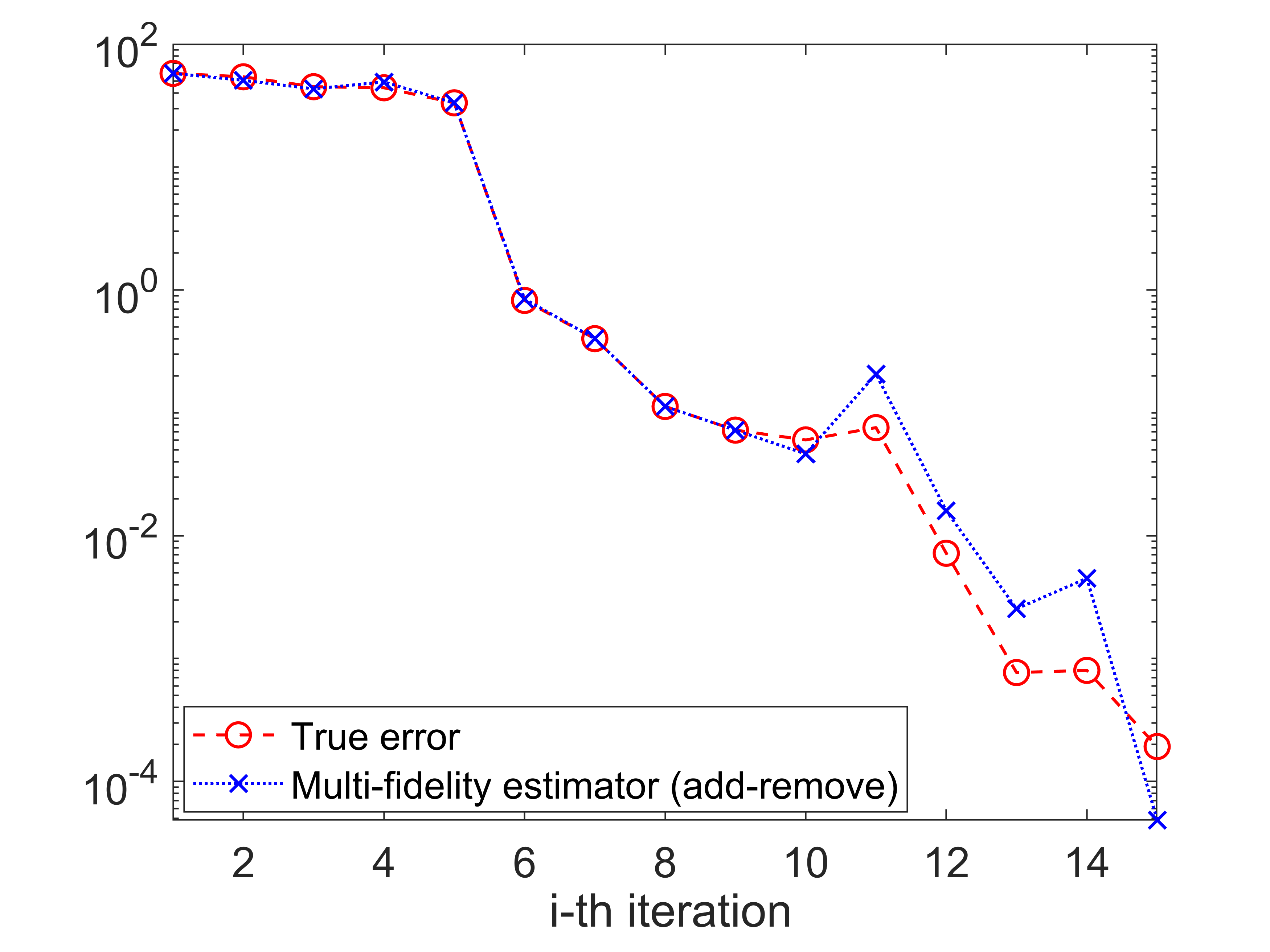}
\caption{Error decay. Left: true error vs multi-fidelity error estimator by only adding samples to $\Xi_c$. Right: true error vs multi-fidelity error estimator by adding and deleting samples to and from $\Xi_c$.}
\label{fig:divider2}
\end{figure}
%
%
%
%
%
%
\subsection{Test 2: results for a model of coplanar microstrips}
\label{subsec:model2}
The second example is a model of a three coplanar microstrips structure shown in Fig. \ref{fig:microstrip}.
The width of the metal strips is $m_w = 0.178$ mm, 
the thickness of metal strips and ground plane is $m_t = 0.035$ mm while 
the left and right wing of the microstrips are $w_d = 3$ mm.
Finally, the length of each strip is $\ell = 5$ cm, the thickness of the dielectric is $d_t = 0.8\,$mm, and the spacing between 2 strips is $s = 0.3\,$mm.
The relative dielectric constant is set to be $\varepsilon_r = 4$ and the conductivity of the metal is assumed to be $\sigma = 5.8 \time 10^7$ S/m.
The six ports, located between the ends of each strip and the ground plane below, are terminated on load resistors $R_{load} = 50 \ \Omega$. The order of the FOM is $n=16,644$, and there are $d=168$ delays. The frequency band of interest is $[0,10]$GHz.

For this model, we take $f_l=1\times10^6, f_h=1\times10^{10}$. We set 30 samples for $\Xi$ in the standard greedy Algorithm~\ref{alg:greedy_delta1pr}, i.e., $|\Xi|=30$. For Algorithm~\ref{alg:greedy_bifidelity} and Algorithm~\ref{alg:greedy_multifidelity}, $|\Xi_c|=10$ or $|\Xi_c|=15$, and $|\Xi_f|=100$. The 1000 samples used for validating the ROM accuracy are generated using the MATLAB function \texttt{linspace}, with $f_l=100$ and the given $f_h$.
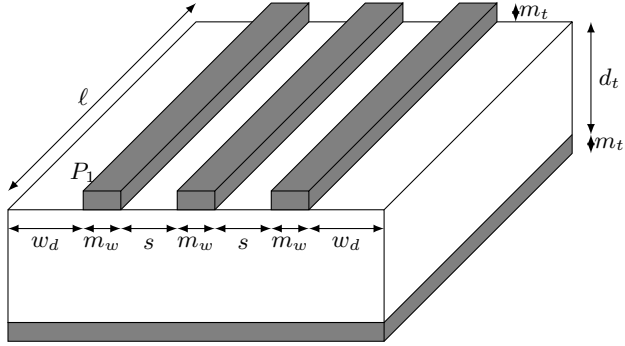
\begin{figure}
\centering
\begin{tikzpicture}[scale=.5]
\fill [gray]  (0,-0.5) rectangle (10,0);
\fill [gray]  (10,-0.5) -- (15,4.5) -- (15,5) -- (10,0) -- (10,-.05);
\draw (0,-0.5) rectangle (10,0);
\draw (10,-0.5) -- (15,4.5) -- (15,5);
\draw (0,0) rectangle (10,3);
\draw (10,0) -- (15,5) -- (15,8) -- (5,8) -- (0,3)
		(15,8) -- (10,3);
\fill [gray]  (3,3) -- (8,8) -- (8,8.5) -- (7,8.5) -- (2,3.5) -- (2,3) -- (3,3);
\fill [gray]  (2,3) rectangle (3,3.5);
\draw (2,3) rectangle (3,3.5);	
\draw (3,3) -- (8,8) -- (8,8.5) -- (3,3.5)
	 (8,8.5) -- (7,8.5) -- (2,3.5);
\fill [gray]  (5.5,3) -- (10.5,8) -- (10.5,8.5) -- (9.5,8.5) -- (4.5,3.5) -- (4.5,3) -- (5.5,3);
\fill [gray]  (4.5,3) rectangle  (5.5,3.5);
\draw (4.5,3) rectangle (5.5,3.5);	
\draw (5.5,3) -- (10.5,8) -- (10.5,8.5) -- (5.5,3.5)
	 (10.5,8.5) -- (9.5,8.5) -- (4.5,3.5);
\fill [gray]  (8,3) -- (13,8) -- (13,8.5) -- (12,8.5) -- (7,3.5) -- (7,3) -- (8,3);
\fill [gray]  (7,3) rectangle (8,3.5);
\draw (7,3) rectangle (8,3.5);
\draw (8,3) -- (13,8) -- (13,8.5) -- (8,3.5)
	 (13,8.5) -- (12,8.5) -- (7,3.5);

\draw [latex-latex] (0,2.5) -- (2,2.5);
\node at (1,2.1) {\footnotesize{$w_d$}};
\draw [latex-latex] (2,2.5) -- (3,2.5);
\node at (2.5,2.1) {\footnotesize{$m_w$}};
\draw [latex-latex] (3,2.5) -- (4.5,2.5);
\node at (3.75,2.1) {\footnotesize{$s$}};
\draw [latex-latex] (4.5,2.5) -- (5.5,2.5);
\node at (5,2.1) {\footnotesize{$m_w$}};
\draw [latex-latex] (5.5,2.5) -- (7,2.5);
\node at (6.25,2.1) {\footnotesize{$s$}};
\draw [latex-latex] (7,2.5) -- (8,2.5);
\node at (7.5,2.1) {\footnotesize{$m_w$}};
\draw [latex-latex] (8,2.5) -- (10,2.5);
\node at (9,2.1) {\footnotesize{$w_d$}};

\draw [latex-latex] (13.5,8) -- (13.5,8.5);
\node at (14,8.25) {\footnotesize{$m_t$}};
\draw [latex-latex] (15.5,8) -- (15.5,5);
\node at (16,6.5) {\footnotesize{$d_t$}};
\draw [latex-latex] (15.5,5) -- (15.5,4.5);
\node at (16,4.75) {\footnotesize{$m_t$}};

\draw [latex-latex] (0,3.5) -- (5,8.5);
\node at (2,6) {\footnotesize{$\ell$}};

\node at (2,4) {\footnotesize{$P_1$}};

\end{tikzpicture}
\caption{Three coplanar microstrips}\label{fig:microstrip}
\end{figure}

The results of the three algorithms are listed in Table~\ref{tab:coplanar1}.
The standard greedy Algorithm~\ref{alg:greedy_delta1pr} takes 19 hours, resulting in a ROM of order $r=132$ with validated error below the tolerance tol. During the greedy iteration, if the small parameter set $\Xi_c$ is enriched only (add only), the greedy algorithm with bi-fidelity error estimation and that with multi-fidelity error estimation converge within the same number of iterations, producing ROMs with the same sizes and validated errors. But the greedy algorithm with multi-fidelity error estimation is almost one hour faster. Similar phenomenon happens to the case ``add-remove''. The greedy algorithm with bi-fidelity error estimation and that with multi-fidelity error estimation also converge within the same number of iterations and construct ROMs with the same sizes and accuracy. The runtimes of both algorithms are much less as compared to their ``add only'' versions. Finally, the greedy algorithm with multi-fidelity error estimation by adding and deleting samples to and from $\Xi_c$ (``add-remove'') is most efficient in terms of both runtime and accuracy. It is more than 3 times faster than the standard greedy algorithm resulting in a speed-up of 4.2x, and produces a ROM with even the smallest validated error. 

We note that using $|\Xi_c|=10$, the ROMs constructed by the bi-fidelity greedy algorithm and the multi-fidelity greedy algorithm with adding the samples only have validated errors larger than the tolerance. If we increase $|\Xi_c|$ from 10 to 15, both algorithms generate ROMs with improved accuracy. The results are presented in Tabel~\ref{tab:coplanar2}. However, the computational time also increases a lot. Again, the multi-fidelity greedy algorithm outperforms the bi-fidelity one w.r.t. both accuracy and runtime. 
In contrast to the results in Tables~\ref{tab:divider1}-\ref{tab:divider2} for the divider model, the results for the coplanar microstrips model in both Tables~\ref{tab:coplanar1}-\ref{tab:coplanar2} show that the bi-fidelity greedy algorithm (``add-remove'') is more accurate than its ``add-only'' version. 
\begin{table}[!h]
\centering
\caption{Three coplanar microstrips: $n=16,644$, $d=168$ delays, tol=0.001, adding/removing a single sample at each iteration.}
\label{tab:coplanar1}
\begin{tabular}{|l|l|l|l|l|}
\hline
Method & Iter. & Runtime (h) & $r$ & Valid.err  \\ \hline
Alg.~\ref{alg:greedy_delta1pr} (standard, $|\Xi|=30$) &11  & 15 & 132 & $8.5\times 10^{-4}$    \\ \hline
Alg.~\ref{alg:greedy_bifidelity} (bi-fidelity, add only, $|\Xi_c|=10$) &11  &6.2 & 132 &0.0033   \\ \hline
Alg.~\ref{alg:greedy_bifidelity} (bi-fidelity, add-remove, $|\Xi_c|=10$)& 11 & 5.3 &132 &$8.2\times 10^{-4}$   \\ \hline
Alg.~\ref{alg:greedy_multifidelity} (multi-fidelity, add only, $|\Xi_c|=10$) &11  &5.3 &132 &0.0033  \\ \hline
Alg.~\ref{alg:greedy_multifidelity} (multi-fidelity, add-remove, $|\Xi_c|=10$)&11  & 4.5 & 132& $8.2\times 10^{-4}$    \\ \hline
\end{tabular}
\end{table}
\begin{table}[!h]
\centering
\caption{Three coplanar microstrips: $n=16,644$, $d=168$ delays, tol=0.001, larger $|\Xi_c|$, adding/removing a single sample at each iteration.}
\label{tab:coplanar2}
\begin{tabular}{|l|l|l|l|l|}
\hline
Method & Iter. & Runtime (h) & $r$ & Valid.err  \\ \hline
Alg.~\ref{alg:greedy_bifidelity} (bi-fidelity, add only, $|\Xi_c|=15$) &11  &10 & 132 &0.0011   \\ \hline
Alg.~\ref{alg:greedy_multifidelity} (multi-fidelity, add only, $|\Xi_c|=15$) &12  &9.3 &144 &$4.4\times 10^{-4}$  \\ \hline
\end{tabular}
\end{table}
\begin{table}[!h]
\centering
\caption{Three coplanar microstrips: $n=16,644$, $d=168$ delays, tol=0.001, adding/removing $n_\textrm{add}=n_\textrm{del}>1$ samples at each iteration.}
\label{tab:coplanar3}
\begin{tabular}{|l|l|l|l|l|}
\hline
Method & Iter. & Runtime (h) & $r$ & Valid.err  \\ \hline
Alg.~\ref{alg:greedy_bifidelity} (bi-fidelity, add-remove, $|\Xi_c|=10$, $n_\textrm{add}=2$)& 10 & 4.7 &120 &0.019   \\ \hline
Alg.~\ref{alg:greedy_bifidelity} (bi-fidelity, add-remove, $|\Xi_c|=10$, $n_\textrm{add}=5$)& 10 & 4.7 &120 &0.019   \\ \hline
Alg.~\ref{alg:greedy_multifidelity} (multi-fidelity, add-remove, $|\Xi_c|=10$, $n_\textrm{add}=2$)&10  & 4.2 & 120& 0.019    \\ \hline
Alg.~\ref{alg:greedy_multifidelity} (multi-fidelity, add-remove, $|\Xi_c|=10$, $n_\textrm{add}=5$)&10  & 4.3 & 120& 0.019  \\ \hline
Alg.~\ref{alg:greedy_multifidelity} (multi-fidelity, add-remove, $|\Xi_c|=15$, $n_\textrm{add}=2$)&13  & 7.6 & 156& 0.0027    \\ \hline
\end{tabular}
\end{table}

Table~\ref{tab:coplanar3} shows the results of the bi-fidelity greedy algorithm and the multi-fidelity greedy algorithm based on adding/removing multiple samples at each iteration. For both cases, i.e., $n_\textrm{add}=n_\textrm{del}=2$ and $n_\textrm{add}=n_\textrm{del}=5$, the algorithms using $|\Xi_c|=10$, converge in 10 iterations, one less iteration than they did with $n_\textrm{add}=n_\textrm{del}=1$ in Table~\ref{tab:coplanar1}, resulting in ROMs with smaller order $r$ but with larger validated errors. If we increase $|\Xi_c|$ to 15, then the multi-fidelity greedy algorithm generates a ROM with reduced error, but takes longer time to converge. The bi-fidelity greedy algorithm behaves similarly and its results for $|\Xi_c|=15$ is not presented to avoid repetition. This example again shows that adding/removing a single parameter at each iteration outperforms the cases with $n_\textrm{add}=n_\textrm{del}>1$, and produces ROMs with desired accuracy.   

In Figure~\ref{fig:coplanar}, we show the {\it important} frequency samples of $f$ selected by the greedy algorithms in Table~\ref{tab:coplanar1}. For the case ``add-remove", we find that the greedy algorithm with bi-fidelity error estimation and the one with multi-fidelity error estimation select the same {\it important} frequency samples. Therefore we only plot one group of samples for both algorithms, see the plot ``bi-(multi-) add-remove" in the figure. For the case ``add-only", both algorithms also select the same {\it important} frequency samples, see the plot ``bi-(multi-) add-only" in the figure. This is in agreement with the results given in Table~\ref{tab:coplanar1} where both algorithms for either case produce the same results. The {\it important} frequency samples selected by the high-fidelity error estimator are mostly different from those selected by the other algorithms. It is seen that the {\it important} frequency samples selected by the (bi-)multi-fidelity estimator could be different from those selected by the high-fidelity estimator. However, both can derive ROMs with good accuracy. 
\begin{figure}[!h]
\centering
\includegraphics[width=75mm]{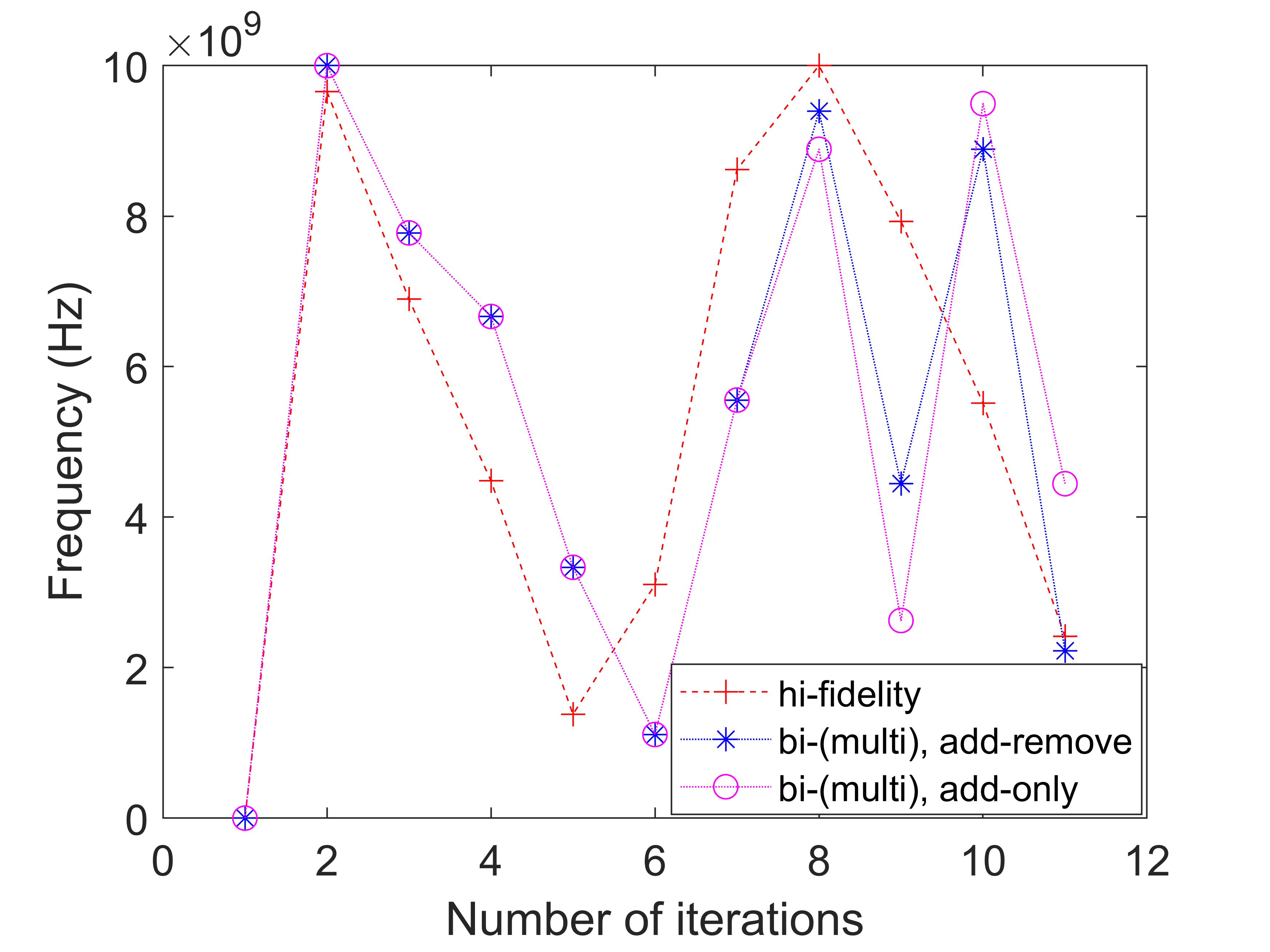}
\caption{{\it Important} parameters selected by the greedy algorithms.}
\label{fig:coplanar}
\end{figure}

The left part of Figure~\ref{fig:coplanar2} gives the error-peak frequencies detected by the multi-fidelity error estimator and the true error, respectively, at each iteration of the greedy algorithm. Those frequencies correspond to the largest values of the error estimator/true error. The error-peak frequency detected by the error estimator at the $i$-th iteration is then selected as the {\it important} frequency sample at the next iteration to update the reduced basis space. From iteration 5, the error-peak frequencies detected by the error estimator are exactly the same as those selected by the true error. This can be explained by the error decay in the right part of the figure. From the $5$-th iteration, the error estimator tightly catches the true error. Although it is less tight at the first 4 iterations, it still follows the overall trend of the error decay and therefore, can still detect reasonable error-peak frequencies. 
This example, once again, supports our theoretical analysis and demonstrates the efficacy of the proposed greedy algorithms with bi-(multi-) fidelity error estimation. 
\begin{figure}[!h]
\centering
\includegraphics[width=75mm]{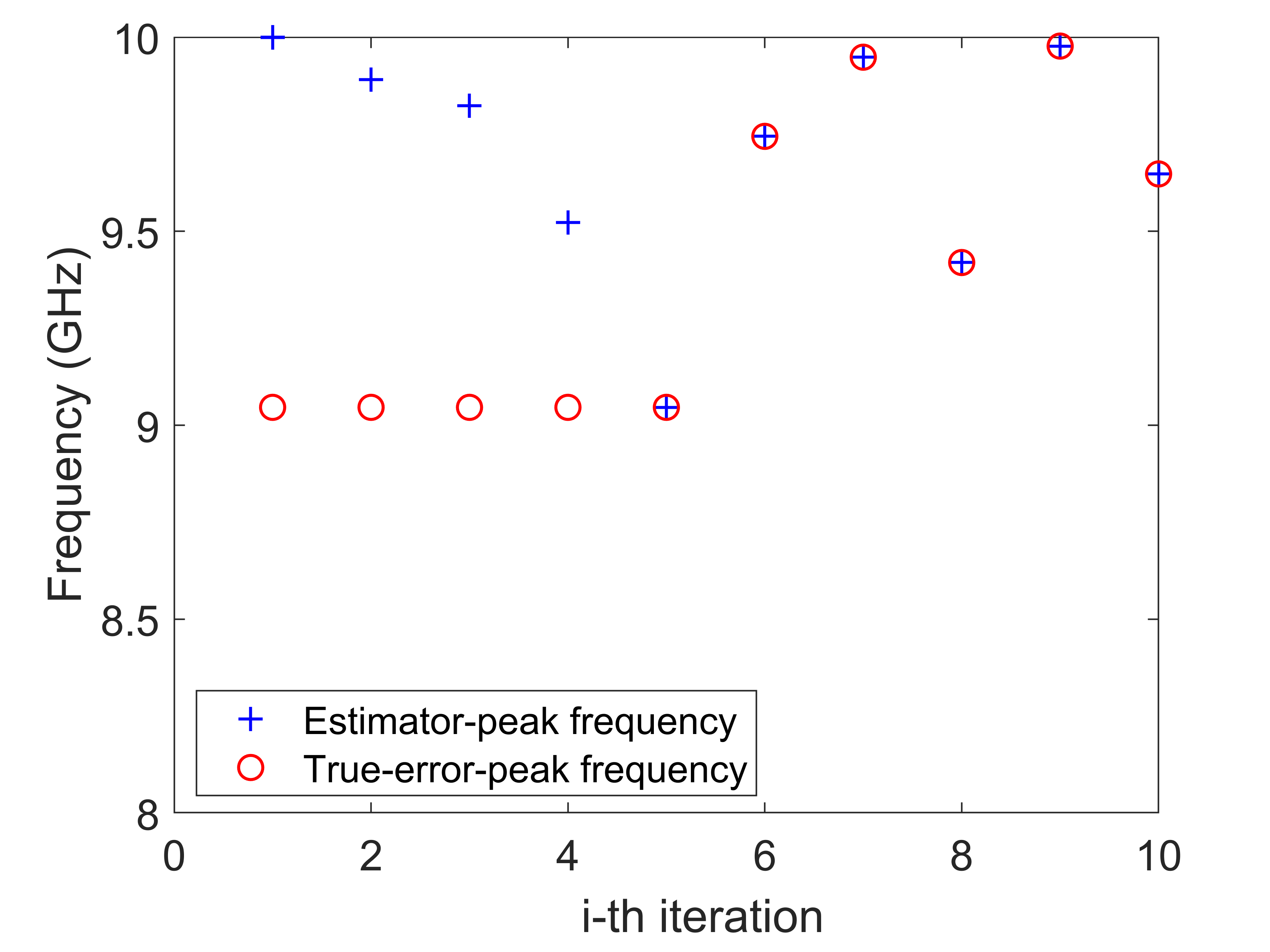}\qquad
\includegraphics[width=75mm]{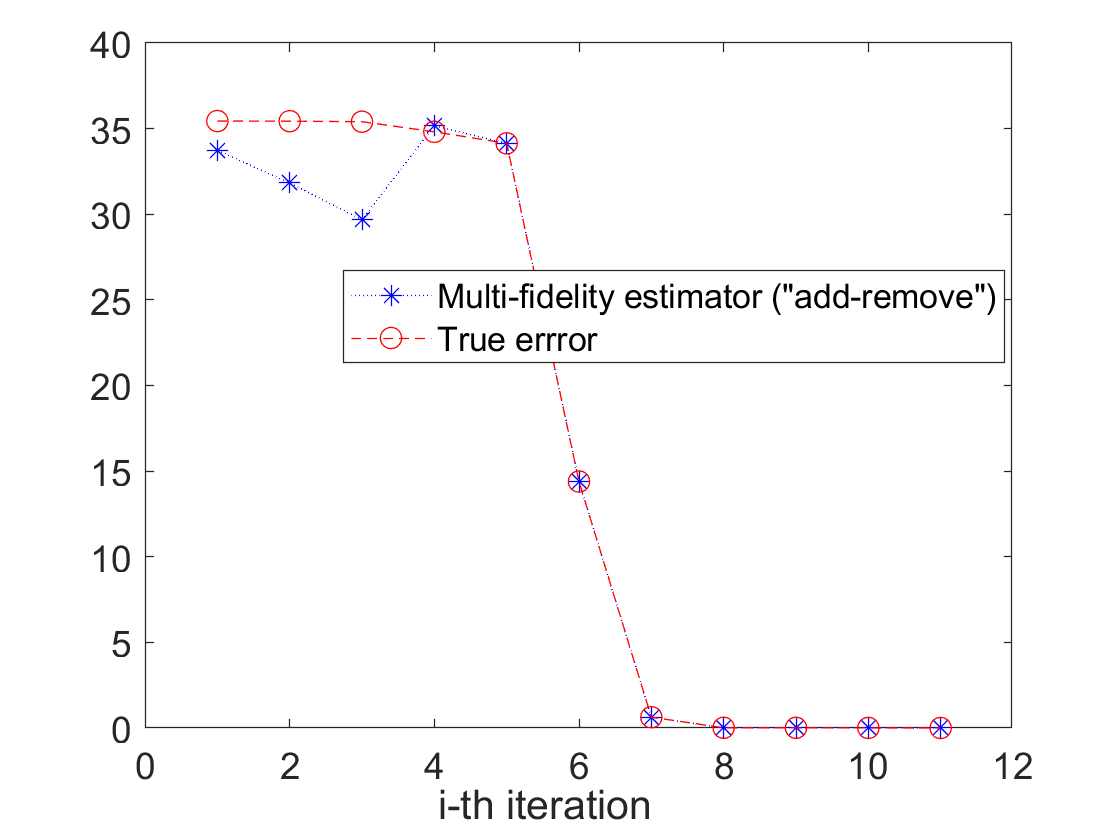}
\caption{Left: Frequencies causing error/estimator peaks. Right: true error vs multi-fidelity error estimator.}
\label{fig:coplanar2}
\end{figure}
\subsection{Test 3: results for a model of microstrip filter}
\label{subsec:model3}
The third example is a model of a microstrip filter. The 3D structure of a microstrip filter is depicted in Fig. \ref{fig:ustrip_filter}. 
The physical dimensions for the geometry of the 3D structure are: 
$w_{z_l}=0.5$ mm,
$w_{z_0}=1.125$ mm,
$w_{z_C}=4$ mm,
$\ell_{z_l}=18.3$ mm,
$\ell_{z_0}=1$ mm,
$\ell_{z_C}=14.1$ mm,
$w=2.4$ cm,
$\ell=2\ell_{z_l} + 2\ell_{z_0} + \ell_{z_C}$,
$t_{m}=100\ \mu$m,
$t_{s}=100\ \mu$m,
$t_{d}=508\ \mu$m.
The two ends of the microstrip are terminated on $50\ \Omega$ resistors. The order of the FOM is $n=12,132$, and there are $d=190$ delays. The interesting frequency band is $[0,5]$GHz. 

We take $f_l=1\times10^5, f_h=5\times10^{9}$ to generate frequency samples in $\Xi_c$ and $\Xi$. We use $|\Xi|=30$ for the standard greedy Algorithm~\ref{alg:greedy_delta1pr}. For Algorithm~\ref{alg:greedy_bifidelity} and Algorithm~\ref{alg:greedy_multifidelity}, $|\Xi_c|=10$, and $|\Xi_f|=100$. The 1000 samples used for computing the validated error are generated using the MATLAB function \texttt{logspace}, with $f_l=10$ and the given $f_h$.
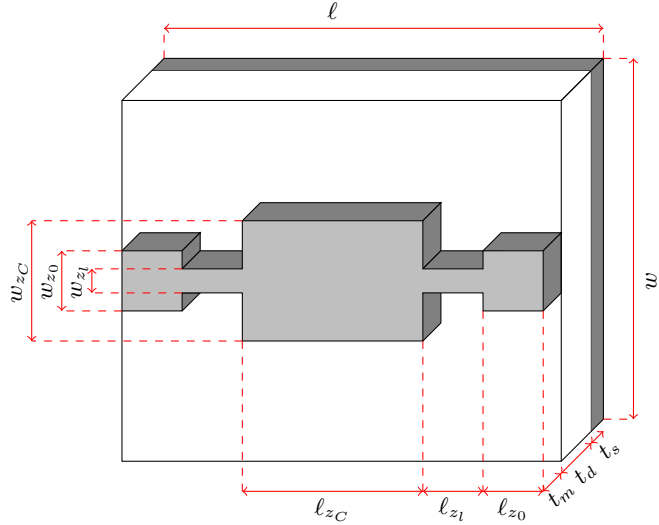
\begin{figure}[htbp]
\centering
\begin{tikzpicture}[scale=.8]
\fill[gray] (1,3.5) -- (1.3,3.8) -- (0.3,3.8) -- (0,3.5);
\draw (1,3.5) -- (1.3,3.8) -- (0.3,3.8) -- (0,3.5);
\fill[gray] (1,3.5) -- (1.3,3.8) -- (1.3,3.5) -- (1,3.2);
\draw (1,3.5) -- (1.3,3.8) -- (1.3,3.5) -- (1,3.2);
\fill[gray](1,3.2) -- (1.3,3.5) -- (2.3,3.5) -- (2,3.2);
\draw  (1,3.2) -- (1.3,3.5) -- (2.3,3.5) -- (2,3.2);
\fill[gray] (5,4) -- (5.3,4.3) -- (2.3,4.3) -- (2,4);
\draw (5,4) -- (5.3,4.3) -- (2.3,4.3) -- (2,4) ;
\fill[gray] (5,4) -- (5.3,4.3) -- ( 5.3, 3.5) -- (5,3.2) ;
\draw (5,4) -- (5.3,4.3) -- ( 5.3, 3.5) -- (5,3.2) ;
\fill[gray] (6,3.2) -- (6.3,3.5) -- (5.3,3.5) -- (5,3.2);
\draw (6,3.2) -- (6.3,3.5) -- (5.3,3.5) -- (5,3.2);
\fill[gray] (6,3.5) -- (6.3,3.8) -- (7.3,3.8) -- (7,3.5);
\draw (6,3.5) -- (6.3,3.8) -- (7.3,3.8) -- (7,3.5);
\fill[gray] (7,2.5) -- (7.3,2.8) -- (7.3,3.8) -- (7,3.5);
\draw  (7,2.5) -- (7.3,2.8) -- (7.3,3.8) -- (7,3.5);

\fill[gray] (1,2.5) -- (1.3,2.8) -- (0.3,2.8) -- (0,2.5);
\draw  (1,2.5) -- (1.3,2.8) -- (0.3,2.8) -- (0,2.5);

\fill[gray] (5,2) -- (5.3,2.3) -- (5.3,3.1) -- (5,2.8);
\draw (5,2) -- (5.3,2.3) -- (5.3,3.1) -- (5,2.8);

\fill[lightgray] (0,2.5) -- (1,2.5) -- (1,2.8) -- (2,2.8) -- (2,2) -- (5,2) -- (5,2.8) -- (6,2.8) -- (6,2.5) -- (7,2.5) -- (7,3.5) -- (6,3.5) -- (6,3.2) -- (5,3.2) -- (5,4) -- (2,4) -- (2,3.2) -- (1,3.2) -- (1,3.5) -- (0,3.5);
\draw (0,0) rectangle (7.3,6);
\draw (0,2.5) -- (1,2.5) -- (1,2.8) -- (2,2.8) -- (2,2) -- (5,2) -- (5,2.8) -- (6,2.8) -- (6,2.5) -- (7,2.5) -- (7,3.5) -- (6,3.5) -- (6,3.2) -- (5,3.2) -- (5,4) -- (2,4) -- (2,3.2) -- (1,3.2) -- (1,3.5) -- (0,3.5);

\draw (0,6) -- (0.5,6.5) -- (7.8,6.5) --  (7.3,6);
\draw (7.3,0) -- (7.8,0.5) -- (7.8,6.5) --  (7.3,6);

\fill[gray] (0.5,6.5) -- (0.7,6.7) -- (8,6.7) -- (7.8,6.5);
\draw (0.5,6.5) -- (0.7,6.7) -- (8,6.7) -- (7.8,6.5);
\fill[gray] (7.8,0.5) -- (8,0.7) -- (8,6.7) -- (7.8,6.5);
\draw (7.8,0.5) -- (8,0.7) -- (8,6.7) -- (7.8,6.5);

\draw [red, dashed] (0,2.5) -- (-1,2.5); 
\draw [red, dashed] (0,3.5) --  (-1,3.5); 
\draw [red,<->]  (-1,2.5) -- (-1,3.5);
\node [rotate = 90] at (-1.2,3) {\footnotesize{$w_{z_0}$}};
\draw [red, dashed] (1,2.8) -- (-0.5,2.8); 
\draw [red, dashed] (1,3.2) --  (-0.5,3.2); 
\draw [red,<->]  (-0.5,2.8) --(-0.5,3.2);
\node [rotate = 90] at (-0.7,3) {\footnotesize{$w_{z_l}$}};
\draw [red, dashed] (2,2) -- (-1.5,2); 
\draw [red, dashed] (2,4) --  (-1.5,4); 
\draw [red,<->]  (-1.5,2) -- (-1.5,4);
\node [rotate = 90] at (-1.7,3) {\footnotesize{$w_{z_C}$}};

\draw [red, dashed] (2,2) -- (2,-0.5);
\draw [red, dashed] (5,2) -- (5,-0.5);
\draw [red, dashed] (6,2.5) -- (6,-0.5);
\draw [red, dashed] (7,2.5) -- (7,-0.5);
\draw [red,<->]  (2,-0.5) -- (5,-0.5);
\node at (3.5,-0.8) {\footnotesize{$\ell_{z_C}$}};
\draw [red,<->]  (6,-0.5) -- (5,-0.5);
\node at (5.5,-0.8) {\footnotesize{$\ell_{z_l}$}};
\draw [red,<->]  (6,-0.5) -- (7,-0.5);
\node at (6.5,-0.8) {\footnotesize{$\ell_{z_0}$}};

\draw [red, dashed] (0.7,6.7) -- (0.7,7.2);
\draw [red, dashed] (8,6.7) -- (8,7.2);
\draw [red,<->]  (0.7,7.2) -- (8,7.2);
\node at (3.5,7.5) {\footnotesize{$\ell$}};
\draw [red, dashed] (8,0.7) -- (8.5,0.7);
\draw [red, dashed] (8,6.7) -- (8.5,6.7);
\draw [red,<->]   (8.5,0.7) -- (8.5,6.7);
\node [rotate = 90] at (8.8,3) {\footnotesize{$w$}};

\draw  [red,<->]  (7,-0.5) -- (7.3, -0.2) ;
\draw  [red,<->] (7.3, -0.2) -- (7.8, 0.3);
\draw  [red,<->] (7.8, 0.3) -- (8, 0.5);
\node [rotate = 45] at (7.3,-0.55) {\footnotesize{$t_m$}};
\node [rotate = 45] at (7.7, -0.2) {\footnotesize{$t_d$}};
\node [rotate = 45] at (8.1,0.2) {\footnotesize{$t_s$}};
\end{tikzpicture}
\caption{Microstrip filter.}
\label{fig:ustrip_filter}
\end{figure}

The results of the high-fidelity greedy algorithm, and the bi-(multi-)fidelity greedy algorithms by adding/removing a single sample at each iteration, are listed in Table~\ref{tab:ustrip1}. All the bi-(multi-)fidelity greedy algorithms produce similar results. The runtime of each is around 1 hour, 3 hours faster than the high-fidelity greedy algorithm. All the ROMs have similar accuracy, with validated errors below the tolerance. 

Table~\ref{tab:ustrip2} further shows the performance of the bi-(multi-)fidelity greedy algorithms by adding and removing multiple samples at each iteration. For this model, all these algorithms behave similarly as they did by  adding/removing a single sample at each iteration. The multi-fidelity greedy algorithm produces ROMs with slightly larger sizes. The ROMs also have larger validated errors, but still fulfill the accuracy requirement. All algorithms converge within 8 iterations, much faster than for the first two examples. This may be due to the much smaller frequency band of interest [0, 5]GHz making the problem much easier to solve and leading to the most efficient performance of all algorithms.

In summary, for all the tested examples, the multi-fidelity algorithm by adding/removing a single sample at each iteration behaves the best w.r.t. both runtime and accuracy. 
\begin{table}[!h]
\centering
\caption{Microstrip filter: $n=12,132$, $d=190$ delays, tol=0.001, adding/removing a single sample at each iteration.}
\label{tab:ustrip1}
\begin{tabular}{|l|l|l|l|l|}
\hline
Method & Iter. & Runtime (h) &$r$ & Valid.err  \\ \hline
Alg.~\ref{alg:greedy_delta1pr} (standard, $|\Xi|=30$) &8  & 2.5 & 32 & $5.6\times 10^{-4}$    \\ \hline
Alg.~\ref{alg:greedy_bifidelity} (bi-fidelity, add only, $|\Xi_c|=10$) &7  &1.1 &28 &$4.6 \times 10^{-4}$   \\ \hline
Alg.~\ref{alg:greedy_bifidelity} (bi-fidelity, add-remove, $|\Xi_c|=10$)& 7 & 1.1 &28 &$4.6 \times 10^{-4}$    \\ \hline
Alg.~\ref{alg:greedy_multifidelity} (multi-fidelity, add only, $|\Xi_c|=10$) &7  &1.0 &28 &$4.6 \times 10^{-4}$   \\ \hline
Alg.~\ref{alg:greedy_multifidelity} (multi-fidelity, add-remove, $|\Xi_c|=10$)&8  & 1.1 & 32& $5.7\times 10^{-4}$    \\ \hline
\end{tabular}
\end{table}
\begin{table}[!h]
\centering
\caption{Microstrip filter: $n=12,132$, $d=190$ delays, tol=0.001, adding/removing $n_\textrm{add}=n_\textrm{del}>1$ samples at each iteration.}
\label{tab:ustrip2}
\begin{tabular}{|l|l|l|l|l|}
\hline
Method & Iter. & Runtime (h) & $r$ & Valid.err  \\ \hline
Alg.~\ref{alg:greedy_bifidelity} (bi-fidelity, add-remove, $|\Xi_c|=10$, $n_\textrm{add}=2$)& 7 & 1.1 &28 &$4.6 \times 10^{-4}$   \\ \hline
Alg.~\ref{alg:greedy_bifidelity} (bi-fidelity, add-remove, $|\Xi_c|=10$, $n_\textrm{add}=5$)& 7 & 1.1 &28 &$4.6 \times 10^{-4}$   \\ \hline
Alg.~\ref{alg:greedy_multifidelity} (multi-fidelity, add-remove, $|\Xi_c|=10$, $n_\textrm{add}=2$)&8  & 1.1 & 32& $9.1 \times 10^{-4}$   \\ \hline
Alg.~\ref{alg:greedy_multifidelity} (multi-fidelity, add-remove, $|\Xi_c|=10$, $n_\textrm{add}=5$)&8  & 1.1 & 32&$9.1 \times 10^{-4}$  \\ \hline
\end{tabular}
\end{table}

\section{Conclusions}
Concepts of bi-fidelity error estimation and multi-fidelity error estimation are proposed in this work. The concept of bi-fidelity error estimation is general and can be applied to any high-fidelity estimator. Although the multi-fidelity error estimation is dependent on the high-fidelity error estimation in consideration, the framework is general to a certain extend and could also be combined with other high-fidelity error estimators. The robustness of the proposed greedy algorithms with bi-fidelity and multi-fidelity error estimation is tested on three large time-delay systems with many delays. Although the standard greedy algorithm converges in a few iterations, the computational complexity in each iteration is high. As a consequence, the runtime is long for such systems. The proposed (bi-)multi-fidelity greedy processes have significantly accelerated the standard greedy algorithm with no loss of accuracy in most cases.

\bibliographystyle{abbrv}
\bibliography{mor,refs}

\end{document}